\documentclass[10pt]{article}
\usepackage[square,authoryear]{natbib}
\usepackage{graphicx}
\usepackage{marsden_article}
\usepackage{stackrel}
\usepackage{ulem}

\usepackage{eucal}
\usepackage{stmaryrd}
\usepackage{mathrsfs}

\usepackage{epstopdf,framed}
\usepackage{pdfsync}
\usepackage{framed}

\DeclareFontFamily{OT1}{pzc}{}
\DeclareFontShape{OT1}{pzc}{m}{it}{<-> s * [1.100] pzcmi7t}{}
\DeclareMathAlphabet{\mathpzc}{OT1}{pzc}{m}{it}

\begin{document}
\newtheorem{remark}[theorem]{Remark}


\title{Variational discretization of the nonequilibrium thermodynamics of simple systems}
\vspace{-0.2in}

\newcommand{\todoFGB}[1]{\vspace{5 mm}\par \noindent
\framebox{\begin{minipage}[c]{0.95 \textwidth} \color{red}FGB: \tt #1
\end{minipage}}\vspace{5 mm}\par}


\author{\hspace{-1cm}
\begin{tabular}{cc}
Fran\c{c}ois Gay-Balmaz &
Hiroaki Yoshimura
\\ CNRS - LMD - IPSL  & School of Science and Engineering
\\ Ecole Normale Sup\'erieure de Paris & Waseda University
\\  24 Rue Lhomond 75005 Paris, France & Okubo, Shinjuku, Tokyo 169-8555, Japan \\ francois.gay-balmaz@lmd.ens.fr & yoshimura@waseda.jp\\
\end{tabular}\\\\
}

\maketitle
\vspace{-0.3in}

\begin{center}
\abstract{In this paper, we develop variational integrators for the nonequilibrium thermodynamics of simple closed systems. These integrators are obtained by a discretization of the Lagrangian variational formulation of nonequilibrium thermodynamics developed in \cite{GBYo2016a}, and thus extend the variational integrators of Lagrangian mechanics, to include irreversible processes. In the continuous setting, we derive the structure preserving property of the flow of such systems. This property is an extension of the symplectic property of the flow of the Euler-Lagrange equations. In the discrete setting, we show that the discrete flow solution of our numerical scheme verifies a discrete version of this property. We also present the regularity conditions which ensure the existence of the discrete flow. We finally illustrate our discrete variational schemes with the implementation of an example of a simple and closed system.}
\vspace{2mm}
\end{center}
\tableofcontents

\section{Introduction}

Nonequilibrium thermodynamics is a phenomenological theory which aims to identify and describe the relations among the observed {\it macroscopic} properties of a physical system and to determine the {\it macroscopic dynamics} with the help of fundamental laws of thermodynamics (e.g. \cite{StSc1974}). The field of nonequilibrium thermodynamics naturally includes macroscopic disciplines such as classical mechanics, fluid dynamics, elasticity, and electromagnetism.

It is well known that the equations of motion of classical mechanics, i.e., the Euler-Lagrange equations can be derived from \textit{Hamilton's variational principle} applied to the action functional associated to the Lagrangian of the mechanical system. One of the many features of the variational formulation is that it admits a discrete version which allows the derivation of structure preserving numerical schemes for the system. Such schemes, called \textit{variational integrators}, (see, \cite{WeMa1997}, \cite{MaWe2001}, \cite{LeMaOrWe2004}) are obtained via a discrete version of Hamilton's principle and are originally based on Moser-Veselov discretizations (see, \cite{Ve1988}, \cite{Ve1991}, \cite{MoVe1991}).
Several extensions of this method have been developed, for example to treat the case of forced mechanical systems (\cite{KaMaOrWe2000}) or nonholonomic mechanical systems (\cite{CoMa2001}, \cite{MLPe2006}).

In \cite{GBYo2016a,GBYo2016b}, we have developed a Lagrangian variational formulation for nonequilibrium thermodynamics which extends Hamilton's principle of classical mechanics by allowing the inclusion of irreversible phenomena in both discrete and continuum systems, i.e., systems with finite and infinite degrees of freedom.
The irreversibility is encoded into a nonlinear nonholonomic constraint given by the expression of the entropy production associated to all the irreversible processes involved.
From a mathematical point of view, the variational formulation of \cite{GBYo2016a,GBYo2016b} may be regarded as a nonlinear generalization of the Lagrange-d'Alembert principle used in nonholonomic mechanics, see e.g., \cite{Bl2003}. In order to formulate the nonholonomic constraint, to each irreversible process is associated a variable called the \textit{thermodynamic displacement} that generalizes the thermal displacement introduced in  \cite{GrNa1991}, following  \cite{He1884}. The introduction of such variables allows the definition of a corresponding variational constraint.

In the present paper, we develop variational integrators for nonequilibrium thermodynamics by discretizing the Lagrangian variational formulation developed in \cite{GBYo2016a}. The resulting numerical schemes are thus extensions of the variational integrators of Lagrangian mechanics that enable to include irreversible phenomena. In the present paper, we restrict our discussions to the case of simple closed systems, i.e., closed systems in which one thermal scalar variable and a finite set of mechanical variables are sufficient to describe entirely the state of the system, though we will be able to develop our discrete theory to handle more general cases including the nonequilibrium thermodynamics of continuum systems.

A key property of variational integrators in Lagrangian mechanics is their symplecticity, meaning that the discrete flow, similarly to the flow of the continuous system, preserves a symplectic form. This ensures an excellent  
long-time energy behavior, see \cite{HaLuWa2006}. When irreversible effects are considered in the dynamics, the symplecticity of the flow may be lost at the continuous level, so there is no hope to discretize the system with a symplectic integrator, in general. In the paper, we shall present a property of the flow $F_t$ of simple closed systems in thermodynamics, which reduces to the symplecticity of the flow in absence of thermal effects.  This property has the form
\[
F_t ^\ast \Omega - \Omega = - \mathbf{d} \int_0^t (F_s^\ast \omega) ds, \quad\text{for all $t$},
\]
where $ \Omega $ is an entropy-dependent symplectic form, $\omega $ is a one-form encoding the effects of friction and temperature, and $ \mathbf{d} $ is the exterior derivative. We then show that our numerical scheme verifies a discrete version of this formula and  therefore it reduces to a symplectic integrator in absence of thermal effects.

The paper is organized as follows. In Section \ref{Section_2} we review the fundamental laws governing the nonequilibrium thermodynamics of macroscopic systems by following the axiomatic formulation of \cite{StSc1974}. Then, we also review the Lagrangian variational formulation of nonequilibrium thermodynamics developed in \cite{GBYo2016a}, which is an extension of Hamilton's principle of classical mechanics that allows the inclusion of irreversible phenomena.
In Section \ref{Section_3}, after recalling some basic facts about variational integrators in Lagrangian mechanics, we propose a discrete version of the variational formulation for nonequilibrium thermodynamics of simple closed systems and deduce a variational integrator for these systems. In Section \ref{Section_4}, we present a property of the flow $F_t$ of simple closed systems in thermodynamics, which reduces to the symplecticity of the flow in absence of thermal effects. Then we show that the discrete flow of our variational integrator verifies a discrete version of this property. We also study the regularity conditions which ensure the existence of the discrete flow. Finally, in Section \ref{Section_5}, we illustrate the implementation of our integrator with an example of a simple system.

\section{Nonequilibrium thermodynamics of simple systems}\label{Section_2}

In this section we first review the fundamental laws governing the nonequilibrium thermodynamics of macroscopic systems.
We follow the axiomatic formulation of thermodynamics developed by Stueckelberg
around 1960 (see, for instance, \cite{StSc1974}), which is well suited for the study of nonequilibrium thermodynamics as a general macroscopic dynamic theory that extends classical mechanics to account for irreversible processes.
Needless to say, it is important to point out that this axiomatic formulation includes the description of systems \textit{out of equilibrium} and is not restricted to the treatment of equilibrium states and transition from one equilibrium state to another. Then, we review the Lagrangian variational formulation of nonequilibrium thermodynamics  from \cite{GBYo2016a}  which is an extension of Hamilton's principle of classical mechanics to allow the inclusion of irreversible phenomena. For brevity, in this paper, we will restrict to the case of simple and closed systems.

\subsection{Fundamental laws of nonequilibrium thermodynamics} 
For the macroscopic description of nonequilibrium thermodynamics, we have the following laws, see \cite{StSc1974}:
\begin{itemize}
\item[(I)]  
\noindent {\bf First law:} For every system, there exists an extensive scalar state function $E$, called {\bfi energy}, which satisfies
\[
\frac{d}{dt} E(t) = P^{\rm ext}_W(t)+P^{\rm ext}_H(t)+P^{\rm ext}_M(t),
\]
where $t$ denotes {\it time}, $ P^{\rm ext}_W(t)$ is the {\it power due to external forces} acting on the mechanical variables of the system, $P^{\rm ext}_H(t)$ is the {\it power due to heat transfer}, and $P^{\rm ext}_M(t)$ is the {\it power due to matter transfer} between the system and the exterior.

\item[(II)] 
\noindent {\bf Second law:} For every system, there exists an extensive scalar state function $S$, called {\bfi entropy}, which obeys the following two conditions.
\begin{itemize}
\item[(a)]  Evolution part:\\
If the system is adiabatically closed, the entropy $S$ is a non-decreasing function with respect to time, i.e., 
\[
\frac{d}{dt} S(t)=I(t)\geq 0,
\]
where $I(t)$ is the {\it entropy production rate} of the system accounting for the irreversibility of internal processes.
\item[(b)] Equilibrium part:\\
If the system is isolated, as time tends to infinity the entropy tends towards a finite local maximum of the function $S$ over all the thermodynamic states $ \rho $ compatible with the system, i.e., 
\[
\lim_{t \rightarrow +\infty}S(t)= \max_{ \rho \; \text{compatible}}S[\rho ].
\]
\end{itemize}
\end{itemize}

In this context, a system is said to be {\bfi closed} if there is no exchange of matter between the system and the exterior, i.e.,  $P^{\rm ext}_M(t)=0$;  a system is said to be {\bfi adiabatically closed} if it is closed and there is no heat exchanges between the system and the exterior, i.e., $P^{\rm ext}_M(t)=P^{\rm ext}_H(t)=0$; and a system is said to be {\bfi isolated} if it is adiabatically closed and there is no mechanical power exchange between the system and the exterior, i.e., $P^{\rm ext}_M(t)=P^{\rm ext}_H(t)=P^{\rm ext}_W(t)=0$.
\medskip

By definition, the evolution of an isolated system is said to be {\bfi reversible} if $I(t) = 0$, namely, the entropy is constant. In general, the evolution of a system is said to be \textit{reversible}, if the evolution of the total isolated system with which it interacts is reversible.

\medskip

In this paper, we only consider simple and closed systems. By definition, a {\bfi simple system}$^{1}$ is a system where one (scalar) thermal variable $S$ and a finite set $(q^{i}, \dot{q}^{i})$ of mechanical variables are sufficient to describe entirely the state of the system, and we assume that there is  {\it no power due to matter transfer} $P^{\rm ext}_M$ between the system and the exterior since the system is closed.

\addtocounter{footnote}{1}
\footnotetext{In \cite{StSc1974} they are called \textit{\'el\'ement de syst\`eme} (French). We choose to use the English terminology \textit{simple system} instead of \textit{system element}. See also \cite{GBYo2016a}.}

\subsection{Variational formulation for nonequilibrium thermodynamics}
Consider a simple closed system described by a mechanical variable $ q \in Q$ and one entropy variable $S \in \mathbb{R}  $. Let $L=L(q, \dot q, S): TQ \times \mathbb{R}  \rightarrow \mathbb{R}$ be the Lagrangian of the system, $F^{\rm ext}:TQ\times \mathbb{R}  \rightarrow T^*Q$  the external force, $F^{\rm fr}:TQ\times \mathbb{R}  \rightarrow T^*Q$ the friction force, and $P^{\rm ext}_H$ the power due to heat transfer between the system and the exterior. The forces are fiber preserving maps, i.e., $F^{\rm fr}(q, \dot q, S), F^{\rm ext}(q, \dot q, S) \in T^*_qQ$, where $T^*_qQ$ denotes the cotangent space to $Q$ at $q$.

A common form for the Lagrangian is
\begin{equation*}\label{special_form}
L(q  , \dot q  , S):=K_{\rm mech}( q  , \dot q )-U(q , S),
\end{equation*} 
where $K_{\rm mech}:TQ \rightarrow \mathbb{R}  $ denotes the kinetic energy of the mechanical part of the system (assumed to be independent of $S$) and $U: Q \times \mathbb{R}  \rightarrow \mathbb{R}$ denotes the {\it potential energy}, which is a function of both the mechanical variable $q$ and the entropy $S$.

\medskip

The variational formulation for the thermodynamics of simple closed systems is defined as follows; see Def. 2.1 in \cite{GBYo2016a}.
\medskip

A curve $(q(t),S(t)) \in Q \times \mathbb{R}$, $t \in [0, T] \subset \mathbb{R}$ is a {\it solution of the variational formulation} if
it satisfies the variational condition 
\begin{equation}\label{LdA_thermo} 
\delta \int_{0}^{T}L(q , \dot q , S)dt +\int_{0}^{T}\left\langle F^{\rm ext}(q, \dot q, S), \delta q\right\rangle dt =0, \quad\;\;\; \textsc{Variational Condition}
\end{equation}
for admissible variations $ \delta q(t) $ and $\delta S(t)$  subject to the constraint
\begin{equation}\label{Virtual_Constraints} 
\frac{\partial L}{\partial S}(q, \dot q, S)\delta S= \left\langle F^{\rm fr}(q , \dot q , S),\delta q \right\rangle,\qquad\qquad\, \textsc{Variational Constraint}
\end{equation}
and if it satisfies the nonlinear nonholonomic constraint 
\begin{equation}\label{Kinematic_Constraints} 
\frac{\partial L}{\partial S}(q, \dot q, S)\dot S  = \left\langle F^{\rm fr}(q, \dot q, S) , \dot q \right\rangle -  P_H^{\rm ext}. \quad \textsc{Phenomenological Constraint}
\end{equation} 

\medskip

Taking variations of the integral in \eqref{LdA_thermo}, integrating by part and using $ \delta q(0)= \delta q(T)=0$, it follows
\[
\int_{0}^{T}\left(  \left\langle  \frac{\partial L}{\partial q}- \frac{d}{dt} \frac{\partial L}{\partial \dot q} +F^{\rm ext}, \delta q \right\rangle + \frac{\partial L}{\partial S}\delta S \right) dt=0,
\]
where the variations $ \delta q $ and $ \delta S$ have to satisfy the variational constraint \eqref{Virtual_Constraints}. Now, replacing $ \frac{\partial L}{\partial S}\delta S $ by the virtual work expression $\left\langle F^{\rm fr}(q, \dot q, S), \delta q \right\rangle $ according to \eqref{Virtual_Constraints} and using the phenomenological constraint, the curve $(q(t),S(t))$ satisfies the following evolution equations for the thermodynamics of the simple closed system
\begin{equation}\label{thermo_mech_equations}
\left\{ 
\begin{array}{l}
\displaystyle\vspace{0.2cm}\frac{d}{dt} \frac{\partial L}{\partial \dot q}- \frac{\partial L}{\partial q}= F^{\rm ext}(q , \dot q, S) +  F^{\rm fr}(q, \dot q, S),\\
\displaystyle\frac{\partial L}{\partial S}\dot S= \left\langle F^{\rm fr} (q , \dot q, S) , \dot q\right\rangle  -  P_H^{\rm ext}.
\end{array} 
\right.
\end{equation} 
Notice that the explicit expression of the constraint \eqref{Kinematic_Constraints}  involves phenomenological laws for the friction force $ F^{\rm fr}$; this is the reason why we refer to it as a {\it phenomenological constraint}. The constraint \eqref{Virtual_Constraints} is called a {\it variational constraint}  since it is a condition on the variations to be used in \eqref{LdA_thermo}, which follows from \eqref{Kinematic_Constraints} by formally replacing the velocity by the corresponding virtual displacement, and by removing the contribution from the exterior of the system. Such a simple correspondence between the phenomenological and variational constraints still holds for the general class of thermodynamical systems considered in \cite{GBYo2016a,GBYo2016b,GB2017}. 

\paragraph{Energy balance law.} The energy associated with $L:TQ \times \mathbb{R}  \rightarrow \mathbb{R}  $ is the function $E:TQ \times \mathbb{R}  \rightarrow \mathbb{R}  $ defined by $E(q, \dot q, S):= \left\langle \frac{\partial L}{\partial \dot q}, \dot q \right\rangle -L(q, \dot q, S)$.
Using the system \eqref{thermo_mech_equations} and defining $P_W^{\rm ext}:= \left\langle F^{\rm ext}, \dot q \right\rangle $, we obtain the {\it energy balance law}\,:
\begin{equation*}\label{ddt_E} 
\frac{d}{dt} E= \left\langle \frac{d}{dt} \frac{\partial L}{\partial \dot q}- \frac{\partial L}{\partial q} , \dot q \right\rangle - \frac{\partial L}{\partial S}\dot S=P_W^{\rm ext}+P_H^{\rm ext},
\end{equation*} 
which is consistent with the first law of thermodynamics.
Notice that {\it energy is preserved when the system is isolated}, i.e., when $P^{\rm ext}_W=P^{\rm ext}_H=P^{\rm ext}_M=0$, consistently with the first law of thermodynamics.

\paragraph{Entropy production.} The temperature is given by minus the partial derivative of the Lagrangian with respect to the entropy, $T=- \frac{\partial L}{\partial S}$, which is assumed to be positive. So the second equation in \eqref{thermo_mech_equations} reads
\[
T \dot S= P_H^{\rm ext}- \left\langle F^{\rm fr}(q, \dot q, S), \dot q\right\rangle .
\]
According to the second law of thermodynamics, for adiabatically closed systems, i.e., when $P^{\rm ext}_H=P^{\rm ext}_M=0$, entropy is increasing. So the friction force $F^{\rm fr}$ must be dissipative, that is $\left\langle F^{\rm fr}(q, \dot q, S), \dot q \right\rangle \leq 0$, for all $(q, \dot q, S) \in TQ \times \mathbb{R}  $. For the case in which the force is linear in velocity, i.e., $F^{\rm fr}(q, \dot q, S)=- \lambda  (q,S) ( \dot q,\_\,)$, where $ \lambda  (q,S):T_qQ  \times T_qQ \rightarrow \mathbb{R}  $ is a {\it two covariant tensor field}, this implies that the symmetric part $ \lambda ^{\rm sym}$ of $ \lambda $ has to be positive. For a simple system, the internal entropy production has the form
\[
I(t)= -\frac{1}{T}\left\langle F^{\rm fr}(q, \dot q, S), \dot q \right\rangle.
\]

\paragraph{Recovering Hamilton's principle.} In absence of the entropy variable and the external force, the constraints  disappear and hence the variational formulation given in equations \eqref{LdA_thermo}--\eqref{Kinematic_Constraints} reduces to Hamilton's principle of Lagrangian mechanics
\begin{equation}\label{HP} 
\delta \int_{0}^{T}L(q , \dot q  )dt =0,
\end{equation} 
for variations $\delta q(t)$ 
vanishing at the endpoints, i.e., $\delta q(0)=\delta q(T)=0$.

\paragraph{Reversibility.} As we recalled earlier, the evolution of an isolated system is said to be reversible if the entropy is constant. In the case of an isolated simple system, in view of the second equation in \eqref{thermo_mech_equations} this means that the evolution is such that 
\begin{equation*}\label{reversibility} 
\left\langle F^{\rm fr}(q(t), \dot q(t), S(t)), \dot q(t) \right\rangle =0.
\end{equation*} 

\section{Discretization of the variational formulation}\label{Section_3}

In this section we first make a brief review of some basic facts about variational integrators in Lagrangian mechanics. Then we propose a discrete version of the variational formulation for nonequilibrium thermodynamics of simple closed systems and deduce a variational integrator for these systems. We also present a condition which ensures the existence of the flow of the integrator and we make several comments on the construction of the constraint.

\subsection{Variational integrators in Lagrangian mechanics} 
Variational integrators are numerical schemes that arise from a discrete version of Hamilton's variational principle \eqref{HP}; see, for instance, \cite{WeMa1997} and \cite{MaWe2001}.
Let $Q$ be the configuration manifold of a mechanical system and let $L:TQ \rightarrow \mathbb{R}  $ be a Lagrangian. Suppose that a time step $h$ has been fixed, denote by $\{t_k =kh\mid k=0,...,N\}$ the sequence of times 
discretizing $[0,T]$, and by $q_d:\{t_k\}_{k=0}^N \rightarrow Q$, 
$q_k:=q_d(t_k)$ the corresponding discrete curve. A discrete Lagrangian $L_d: Q \times Q \rightarrow \mathbb{R}  $ is an approximation of the time integral of the continuous Lagrangian between two consecutive configurations $q_k$ and $q_{k+1}$
\begin{equation}\label{L_d}
L_d( q_k , q_{k+1})\approx
\int_{t _k }^{t_{k+1}}L( q(t), \dot q(t)) dt,
\end{equation} 
where  $q_{k}=q(t_k )$ and $q_{k+1}=q(t_{k+1})$. Equipped with such a discrete Lagrangian, one can now formulate a discrete version of Hamilton's principle \eqref{HP} according to
\[
\delta \sum_{k=0}^{N-1}L_d(q _k , q_{k+1})=0,
\]
for variations $ \delta q _k $ vanishing at the endpoints. Thus, if we denote $D _i $ the partial derivative with respect to the $i^{th}$ variable, three consecutive configuration variables $q_{k-1}, q_k , q_{k+1}$ must verify the discrete analogue of the Euler-Lagrange equations:
\begin{equation}\label{DEC} 
D_2L_d(q_{k-1}, q_k)+D_1L_d(q_k, q_{k+1})=0.
\end{equation} 
These \textit{discrete Euler-Lagrange equations} define, under appropriate conditions, an integration scheme which solves for $q_{k+1}$, knowing the two previous configuration variables $q_{k-1}$ and $q_k$.

A discrete Lagrangian $L_d$ is called \textit{regular} if the following maps, called \textit{discrete Legendre transforms}, are local diffeomorphisms: 
\begin{equation}\label{LT_regular}
\begin{aligned} 
&\mathbb{F} ^+L_d: Q \times Q \rightarrow T^*Q, \quad \mathbb{F} ^+L_d( q _0 , q _1 )= (q_1, D_2L_d( q _0 , q _1 )) \in T^*_{q _1 }Q\\
&\mathbb{F} ^-L_d: Q \times Q \rightarrow T^*Q, \quad \mathbb{F} ^+L_d( q _0 , q _1 )= (q_0, -D_1L_d( q _0 , q _1 )) \in T^*_{q _0 }Q.
\end{aligned}
\end{equation} 
In fact it is enough to prove that one of these maps is a local diffeomorphism. This turns out to be equivalent to the invertibility of the matrix $D_1D_2L_d( q _0 , q _1 )$ for all $ q _0 , q _1 $.

Under the regularity hypothesis, the scheme \eqref{DEC} yields a well-defined discrete flow $F_{L_d}: Q\times Q \to Q \times Q;\; (q_{k-1},q_k)\mapsto(q_k, q_{k+1})$ that is symplectic:
\begin{equation}\label{symplectic_integrator} 
F_{L_d} ^\ast \Omega _{L_d}= \Omega _{L_d},
\end{equation} 
where the symplectic form $ \Omega _{L_d}:= (\mathbb{F} ^\pm L_d) ^\ast \Omega _{\rm can}$ is defined with respect to either $ \mathbb{F}  ^+L_d$ or $ \mathbb{F}  L_d ^-$.

\medskip

External forces can be added using a discrete version of the Lagrange-d'Alembert
principle in a similar manner, see \cite{MaWe2001}.

\subsection{Variational integrators for the thermodynamics of simple systems} 
Let us first extend the concept of discrete Lagrangian \eqref{L_d} from mechanics to the nonequilibrium thermodynamics of simple closed systems described by a mechanical variable $ q \in Q$ and one entropy variable $S \in \mathbb{R} $. 

\begin{definition} Consider a simple closed system with Lagrangian $L=L(q, \dot q, S): TQ \times \mathbb{R}  \rightarrow \mathbb{R}$, suppose that a time step $h$ has been fixed, and denote by $\{t_k =kh\mid k=0,...,N\}$ the sequence of times 
discretizing $[0,T]$. A {\bfi discrete Lagrangian} is a function
\[
L_d : ( Q \times Q ) \times ( \mathbb{R}  \times\mathbb{R}  ) \rightarrow \mathbb{R},  
\]
which is an approximation of the time integral of $L$ between two consecutive states $(q_k, S_k)$ and $(q_{k+1}, S_{k+1})$:
\[
L_d( q _k, q_{k+1}, S _k , S_{k+1})\simeq \int_{t_k}^{t_{k+1}}L( q(t), \dot q(t), S(t)) dt,
\]
where $q(t_i)= q_i $ and $S(t_i)=S_i$, for $i=k, k+1$.
\end{definition}

One example of such a discrete Lagrangian, when $Q$ is a vector space, may be given by
\[
L_d( q _k, q_{k+1}, S _k , S_{k+1}):=\frac{h}{2}\left[ L\left(q_{k},\frac{q_{k+1}-q_{k}}{h}, S_{k}\right) +L\left(q_{k+1},\frac{q_{k+1}-q_{k}}{h}, S_{k+1}\right) \right].
\]

Similarly, we define the discrete analogue of external and friction forces as follows.

\begin{definition}
Consider an external force $F^{\rm ext}:TQ\times \mathbb{R}  \rightarrow T^*Q$ and a friction force $F^{\rm fr}:TQ\times \mathbb{R}  \rightarrow T^*Q$, 
which are fiber preserving maps, i.e., $F^{\rm fr}(q, \dot q, S), F^{\rm ext}(q, \dot q, S) \in T^*_qQ$. We define {\bfi discrete friction forces} and {\bfi discrete exterior forces} to be maps
\[
F^{\rm fr -}, F^{\rm fr +}, F^{\rm ext -}, F^{\rm ext +}:  ( Q \times Q ) \times ( \mathbb{R}  \times\mathbb{R}  )  \rightarrow T^*Q,
\]
such that the following approximation holds
\begin{align*} 
&\left\langle F^{\rm fr -}( q _k, q_{k+1}, S _k , S_{k+1}) , \delta q _k \right\rangle + \left\langle F^{\rm fr +}( q _k, q_{k+1}, S _k , S_{k+1}) , \delta q _{k+1}\right\rangle \\
&\hspace{7cm} \simeq \int_{t_k}^{t_{k+1}}\left\langle F^{\rm fr}(q(t), \dot q(t), S(t)) , \delta q(t) \right\rangle,
\end{align*} 
similarly for $F^{\rm ext \pm}$, $F^{\rm ext}$,
where $q(t_i)= q_i $, $S(t_i)=S_i$, $\delta q(t_i)= \delta q_i $, and $\delta S(t_i)= \delta S_i$, for $i=k, k+1$.

These discrete forces are required to be fiber preserving in the sense that
\[
\pi _Q \circ F^{\rm fr \pm}= \pi _{Q}^\pm, \quad \pi _Q \circ F^{\rm ext \pm}= \pi _{Q}^\pm,
\]
where $\pi _Q: T^*Q \rightarrow Q$ is the cotangent bundle projection and  $ \pi _Q^-, \pi _Q^+:( Q \times Q ) \times ( \mathbb{R}  \times\mathbb{R}  ) \rightarrow Q$ are defined by $ \pi _Q ^-(q_0,q_1, S_0, S_1)=q_0$ and $\pi _Q^+(q_0,q_1, S_0, S_1)=q_1$.
\end{definition}


\paragraph{Construction of the constraint.} For the case of nonholonomic mechanics with linear constraint, the discrete constraint can be constructed from a finite difference map, see \cite{CoMa2001} and \cite{MLPe2006}.
We shall extend this construction to our nonlinear situation and with the entropy variable.
\medskip

Following \cite{MLPe2006}, a finite difference map $ \varphi _Q$ on a manifold $Q$ is a diffeomorphism
\[
\varphi_Q :N_0( \Delta_{Q} ) \rightarrow T_0Q,
\]
where $N_0( \Delta_{Q} )$ is a neighborhood of the diagonal $ \Delta_{Q} $ in $Q \times Q$ and $T_0Q$ is a neighborhood of the zero section of $TQ$, which satisfies the following conditions:\\
1. $ \varphi _Q( \Delta_{Q} )$ is the zero section of $TQ$;\\
2. $\tau ( \varphi _Q(N_0( \Delta_{Q} )))=Q$;\\
3. $ \tau (\varphi_{Q} ( q,q ))= q$.  
\medskip

All three conditions can be equivalently described as: $ \varphi_{Q} ( q,q )=0_q$.

\begin{definition}
Taking two finite difference maps
\[
\varphi _Q:N_0( \Delta_{Q} ) \rightarrow T_0Q \quad\text{and}\quad \varphi _ \mathbb{R}  :N_{0}( \Delta_{\mathbb{R}} ) \rightarrow T_0\mathbb{R},
\]
we define the {\bfi finite difference map} $\varphi =\varphi _Q \times \varphi _ \mathbb{R}: N_0( \Delta_{Q \times \mathbb{R}} ) \rightarrow T_0(  Q \times \mathbb{R}  ) $ by
\begin{equation}\label{FDM} 
\varphi ( q _k ,q_{k+1}, S _k , S_{k+1})= \left( \varphi _Q(q _k ,q_{k+1}), \varphi _ \mathbb{R}  (S _k , S_{k+1})\right), 
\end{equation} 
where the neighborhoods are  $N_0( \Delta_{Q \times \mathbb{R}} )\cong N_0( \Delta_{Q} ) \times N_0( \Delta_{\mathbb{R}  } )$, $\Delta_{Q \times \mathbb{R}}= \Delta_{Q}  \times \Delta_{\mathbb{R}  } $, and $T_0(  Q \times \mathbb{R}  ) \cong T_{0}Q \times T_{0}\mathbb{R}$.
\end{definition}

Recall that, in the continuous setting, the phenomenological constraint is the subset  $C_K \subset T(Q \times \mathbb{R}  )$ defined by
\begin{equation}\label{CK} 
(q, \dot q, S, \dot S) \in C_K \Longleftrightarrow \frac{\partial L}{\partial S} (q, \dot q, S) \dot S-  \left\langle F^{\rm fr}(q, \dot q, S) ,  \dot q  \right\rangle =0,
\end{equation} 
where we assumed $P_{H}^{\rm ext}=0$ for simplicity. Notice that for any physically relevant Lagrangian (see also Assumption II in \eqref{physical_assumption} below), the function $P:T(Q \times \mathbb{R}  )\rightarrow \mathbb{R}  $ defined by 
\begin{equation}\label{function_P} 
P(q, \dot q, S, \dot S):=\frac{\partial L}{\partial S} (q, \dot q, S) \dot S- \left\langle   F^{\rm fr}(q, \dot q, S) ,  \dot q \right\rangle,
\end{equation} 
is a {\it submersion}, since $ \frac{\partial P}{\partial \dot S}= \frac{\partial L}{\partial S}(q, \dot q, S)\neq 0$, being minus the temperature. 
Thus $C_K$ is a codimension one submanifold of $T( Q\times \mathbb{R}  )$. Notice also that the zero section is included in $C_K$.

In order to formulate the discrete version of the phenomenological constraint, we need to define a discrete version $C_K ^d \subset (Q \times Q) \times (\mathbb{R}  \times \mathbb{R})$ of the submanifold $C _K \subset T(Q \times \mathbb{R}  )$. Such a discrete version is written with the help of a function $P_d:(Q \times Q) \times ( \mathbb{R}  \times \mathbb{R}  )\rightarrow \mathbb{R}$ as
\begin{equation}\label{CKd}
C_K^d=\{( q _0 , q _1 , S _0 , S _1 ) \in (Q \times Q) \times (\mathbb{R}  \times \mathbb{R}) \mid P_d( q _0 , q _1 , S _0 , S _1 )=0\}.
\end{equation} 
In the definition below, we present a way to construct $C_K ^d $ from a given finite difference map. We will then show how to construct both $C_K^d$ and $L_d$ in a consistent way.

\begin{definition}
Given the constraint in \eqref{CK} and a finite difference map $\varphi $ in \eqref{FDM}, the associated {\bfi discrete constraint} is defined by
\begin{equation}\label{def_CKd} 
C_K^d:= \varphi ^{-1} ( C_K \cap T_0(Q \times \mathbb{R}  )) \subset ( Q \times Q ) \times ( \mathbb{R}  \times\mathbb{R}  ).
\end{equation}
In this case the function $P_d$ in \eqref{CKd} is obtained by composing the function $P: T(Q\times\mathbb{R}) \to \mathbb{R}$ in \eqref{function_P} with the finite difference map $\varphi$. 
\end{definition}
 
It is possible to construct both the discrete phenomenological constraint and the discrete Lagrangian in a consistent way. Indeed, suppose that a finite difference map $\varphi: ( Q \times Q ) \times ( \mathbb{R}  \times\mathbb{R}  ) \to T(Q \times \mathbb{R})$ is given, then one can construct $C_K ^d $ as in \eqref{def_CKd} and $L_d$ as
\begin{equation}\label{construction_L_d} 
L_d:= h L\circ \pi\circ\varphi ,
\end{equation}
where we recall $h$ is the time step and $\pi: T(Q \times \mathbb{R})\cong TQ \times T\mathbb{R} \to TQ \times \mathbb{R}$ is the canonical projection. This formula can be interpreted in two ways. On one hand, as $L_d= h \tilde L \circ \varphi $, where $\tilde{L}:=\pi^{\ast}L$ is the lifted Lagrangian on $T(Q \times \mathbb{R})$, while it can be written as $L_d=h L \circ \Psi $, where we define the \textit{discretizing map} $\Psi$ by $\Psi:= \pi \circ \varphi: ( Q \times Q ) \times ( \mathbb{R}  \times\mathbb{R}  )  \to TQ \times \mathbb{R}$.

\begin{remark}{\rm We will show that the construction of both $L_d$ and $C_K ^d $ from a unique finite difference map $ \varphi $ is not needed to obtain the structure preserving properties in \S\ref{Section_5}. One can choose a finite difference map $ \varphi $ and a discretizing map $ \Psi $ which are not necessarily related through $ \Psi = \pi \circ\varphi $. For example, in nonholonomic mechanics (linear case), there are examples of integrators in which $C_K^d$ and $L_d$ are not constructed from the same finite difference mapping, but which perform extremely well, see (4.18) in \cite{MLPe2006}.}
\end{remark}

\begin{definition}
By analogy with the continuous variational constraint \eqref{Virtual_Constraints}, we define the {\bfi discrete variational constraint} by imposing the following constraint on $ \delta q_k$ and $ \delta S_k$ as
\begin{equation}\label{VC} 
\begin{aligned} 
&D_3 L_d( q _k, q_{k+1}, S _k , S_{k+1}) \delta S_k + D_4 L_d( q _k, q_{k+1}, S _k , S_{k+1}) \delta S_{k+1}\\
& \qquad\qquad = \left\langle  F^{\rm fr -}( q _k, q_{k+1}, S _k , S_{k+1}) , \delta q _k  \right\rangle +  \left\langle F^{\rm fr +}( q _k, q_{k+1}, S _k , S_{k+1}) , \delta q _{k+1} \right\rangle.
\end{aligned}
\end{equation}
\end{definition}

\medskip

\begin{definition}[\textbf{Discrete variational formulation for the nonequilibrium thermodynamics of simple systems}]\label{discrete_VP} Given a discrete Lagrangian $L_d$, discrete friction forces $F^{\rm fr \pm}$,  external forces $F^{\rm ext \pm}$, and a discrete phenomenological constraint $C_K ^d $, a discrete curve $(q_d,S_d)=\{(q_{k},S_{k})\}_{k=0}^{N}$ is a solution of the variational formulation if it satisfies the {\bfi discrete variational condition}
\begin{align*} 
&\delta \sum_{k=0}^{N-1} L_d( q _k, q_{k+1}, S _k , S_{k+1})\\
& \qquad + \sum_{k=0}^{N-1}\left( \left\langle F^{\rm ext -}( q _k, q_{k+1}, S _k , S_{k+1}), \delta q _k  \right\rangle + \left\langle  F^{\rm ext +}( q _k, q_{k+1}, S _k , S_{k+1}) , \delta q _{k+1}  \right\rangle \right) =0,
\end{align*} 
for variations satisfying the {\bfi discrete variational constraint} \eqref{VC} and where the discrete curve $(q_{d},S_{d})=\{(q_{k},S_{k})\}_{k=0}^{N}$ is subject to the {\bfi discrete phenomenological constraint}
\[
(q _k, q_{k+1}, S _k , S_{k+1}) \in C_K ^d.
\]
\end{definition} 

\medskip

A direct application of this variational formulation yields the following result.

\begin{theorem} A discrete curve $(q_{d},S_{d})=\{(q_{k},S_{k})\}_{k=0}^{N}$ is a solution of the variational formulation if and only if it satisfies the following {\bfi discrete evolution equations}:
\begin{equation}\label{discrete_GBYo}
\left\{
\begin{array}{l}
\vspace{0.2cm} D_1L_d(q _k , q_{k+1}, S_k, S_{k+1}) + D_2L_d(q _{k-1} , q_k, S_{k-1}, S_k)\\
\vspace{0.2cm} \qquad \qquad \qquad \qquad \quad + (F^{\rm fr -}+ F^{\rm ext -}) ( q _k, q_{k+1}, S _k , S_{k+1})\\
\vspace{0.2cm} \qquad \qquad \qquad \qquad \quad + (F^{\rm fr +}+  F^{\rm ext +})( q _{k-1}, q_k, S _{k-1} , S_k)=0,\\[2mm]
(q _k, q_{k+1}, S _k , S_{k+1}) \in C_K ^d.
\end{array}
\right.
\end{equation}
\end{theorem}

\color{black}

\paragraph{Discrete flow map.} By applying the implicit function theorem, we see that if the following matrix
\begin{equation}\label{regularity_criteria} 
\left[
\begin{array}{ll}
D_2 D_1 L_d(r) +D_2 F_d^- (r)& D_4D_1 L_d (r)+ D_4 F_d^- (r) \\
D_2P_d(r) & D_4P_d(r) 
\end{array}
\right]
\end{equation} 
is invertible for all $r:=(q_0, q_1,S_0, S_1)$, where we set $F_d^-:=F^{\rm fr-}+F^{\rm ext-}$, then the scheme \eqref{discrete_GBYo} yields a well-defined {\bfi discrete flow}
\begin{equation}\label{discrete_flow}
F_{L_d}: (q _k , q_{k+1}, S_k, S_{k+1})  \in C_K ^d \mapsto (q _{k+1} , q_{k+2}, S_{k+1}, S_{k+2})  \in C_K ^d.
\end{equation} 
It is easy to check that a matrix of the form \eqref{regularity_criteria} is invertible if and only if $D_4P_d(r) \neq 0$ and the matrix
\begin{equation}\label{criteria_rewritten} 
D_2 D_1 L_d(r)+D_2 F_d^- (r)- \frac{1}{D_4P_d(r)} \left(D_4D_1 L_d (r)+ D_4 F_d^- (r)  \right) D_2P_d(r)
\end{equation} 
is invertible. This criteria generalizes to the case of thermodynamics, the regularity criteria of the discrete Lagrangian of discrete mechanics, namely the condition that $D_2 D_1 L_d(q_0, q_1)$ is invertible for all $(q_0,q_1)$, see \eqref{LT_regular}, which may be recovered from \eqref{criteria_rewritten} when the entropy variable and the forces are absent.

\pagebreak

\section{Structure preserving properties}\label{Section_4} 
In this Section, we present a property of the flow $F_t$ of a simple and closed system which reduces to the symplecticity of the flow in absence of thermal effects. Then we show that the discrete flow of our numerical integrator verifies a discrete version of this property.

\subsection{Thermodynamics of simple systems - continuous case}

It is well known that when the Lagrangian $L:TQ \rightarrow \mathbb{R}  $ of a mechanical system is regular, then the flow $F_t: TQ \to TQ$ of the Euler-Lagrange equations preserves the symplectic form $\Omega _L=(\mathbb{F}  L) ^\ast \Omega _{\rm can}$ called the {\it Lagrangian two-form} on $TQ$:
\begin{equation}\label{symplecticity} 
F_t^\ast \Omega _L= \Omega _L.
\end{equation} 
In order to formulate the extension of this property to the case of the thermodynamics of simple systems, we first make below some definitions and assumptions concerning the Lagrangian function in thermodynamics.

\paragraph{Regularity and assumptions on the Lagrangian.}
 Given a Lagrangian $L: TQ \times \mathbb{R}  \rightarrow \mathbb{R}  $, the {\it Legendre transform} is defined by
\[
\mathbb{F} L: TQ \times \mathbb{R}  \rightarrow T^*Q, \quad \mathbb{F}  L(q,v,S):= \Big( q, \frac{\partial L}{\partial v}(q,v,S)\Big).
\]
The only difference with the standard case in mechanics is the dependence on $S$. By definition, we say that the Lagrangian $L(q,v,S)$ is \textit{regular} if and only if {\it for each $S$ fixed}, the map
\[
(q,v) \in T Q \mapsto \mathbb{F}  L(q,v,S) \in  T^*Q
\]
is a local diffeomorphism. One easily checks that this is equivalent to the invertibility of the matrix $ \frac{\partial ^2 L}{\partial v ^i v ^j }(q,v,S)$ for all $(q,v,S)$.
\medskip

We define the following two Lagrangian forms on $TQ \times \mathbb{R}$, namely, the {\it Lagrangian one-form}
\[
\Theta _L(q,v,S):= (\mathbb{F}  L) ^\ast\Theta _{\rm can}= \frac{\partial L}{\partial v}(q,v,S) \mbox{d}q 
\]
and the {\it Lagrangian two-form}
\[
\Omega _L:= - \mathbf{d}\Theta _L \in \Omega ^2 (TQ \times \mathbb{R}  ),
\]
which reads locally
\[
\Omega _L(q,v,S)=  \frac{\partial ^2 L}{\partial v ^i \partial q ^j } (q,v,S) {\mbox{d}q } ^i \wedge {\mbox{d}q }^j +  \frac{\partial ^2 L}{\partial v ^i \partial v ^j } (q,v,S) {\mbox{d}q } ^i \wedge  {\mbox{d}v}^j +  \frac{\partial ^2 L}{\partial v ^i \partial S }  (q,v,S) {\mbox{d}q } ^i \wedge {\mbox{d}S}.
\]
In absence of the entropy variable, these forms recover the usual Lagrangian forms on $TQ$ defined in Lagrangian mechanics.

\medskip

We now write two physical assumptions made on the Lagrangian $L:TQ \times\mathbb{R} \rightarrow \mathbb{R}$.
 
\begin{itemize}
\item{}{\underline{Assumption I}:} 
A first physical restriction on the Lagrangian is the following assumption
\begin{equation}\label{assumption} 
\frac{\partial ^2 L}{\partial v ^i \partial S}=0, 
\end{equation}
which means that the temperature $T= - \frac{\partial L}{\partial S}$ does not depend on $v $ or, equivalently, the momentum $p=\frac{\partial L}{\partial v}$ does not depend on $S$. In other words,
\[
\frac{\partial L}{\partial S}(q,v,S)= \frac{\partial L}{\partial S}(q,S) \quad\text{and}\quad\frac{\partial L}{\partial v}(q,v,S)= \frac{\partial L}{\partial v}(q,v).
\]
It follows from Assumption I \eqref{assumption} that the Lagrangian is necessarily of the form
\[
L(q,v,S)= K( q  , v)-U(q , S),
\]
for two functions $K:TQ \rightarrow \mathbb{R}  $ and $U: Q \times\mathbb{R} \rightarrow \mathbb{R}  $. Under Assumption I, the Lagrangian two-form reads
\[
\Omega _L(q,v,S)=\frac{\partial ^2 L}{\partial v ^i \partial q ^j } (q,v,S) {\mbox{d}q } ^i \wedge {\mbox{d}q }^j +  \frac{\partial ^2 L}{\partial v ^i \partial v ^j }(q,v,S)  {\mbox{d}q }^i \wedge  {\mbox{d}v }^j.
\]
In this case, $ \Omega _L$ can be seen as a $S$-dependent two-form on $TQ$. Moreover, $ \Omega _L$ is symplectic on $TQ$, {\it for each $S$ fixed}, if and only if the Lagrangian is regular.
\medskip

\item{}{\underline{Assumption II}:} 
Any physical Lagrangian must verify the condition
\begin{equation}\label{physical_assumption} 
\frac{\partial L}{\partial S}(q,v,S) <0,\quad \text{for all $(q,v,S)$}, 
\end{equation}
since $\frac{\partial L}{\partial S}=-T$ is identified with minus the temperature. If Assumption I  \eqref{assumption} is verified, then Assumption II \eqref{physical_assumption} reads simply
\[
\frac{\partial U}{\partial S}(q,S) >0,\quad \text{for all $(q,S)$}.
\] 
\end{itemize}

\medskip

\paragraph{Structure preserving property.} 
Recall that given a Lagrangian $L:TQ \times \mathbb{R} \rightarrow \mathbb{R}  $, and the forces $F^{\rm fr}, F^{\rm ext}:TQ \times \mathbb{R}  \rightarrow T^*Q$, the evolution equations are given by the system \eqref{thermo_mech_equations}, rewritten here for the curve $(q(t), v(t), S(t))$ as
\begin{equation}\label{simple_system} 
\left\{ 
\begin{array}{l}
\vspace{0.2cm}\displaystyle\frac{d}{dt}\frac{\partial L}{\partial \dot q}( q(t), v(t), S(t)) - \frac{\partial L}{\partial q} (q(t),v(t), S(t))\\
\vspace{0.2cm}\qquad \qquad \qquad  \qquad \qquad = F^{\rm ext} (q(t), v(t), S(t))+ F^{\rm fr} (q(t),v(t), S(t)),\\
\vspace{0.2cm}\displaystyle\frac{\partial L}{\partial S} (q(t), v(t), S(t)) \dot S (t) =  F^{\rm fr} (q(t), v(t), S(t) )\cdot v(t)   -P_H^{\rm ext} (t),\\
\dot q(t)= v(t),
\end{array}
\right.
\end{equation} 
where $\dot q(t)=\frac{dq}{dt}$.
We assume that the Lagrangian is regular and that the physical assumptions \eqref{assumption} and \eqref{physical_assumption} are verified. In this case, one observes that \eqref{simple_system} gives a well-defined first order ordinary differential equation for the curve $(q(t), v(t), S(t))$ and, therefore, a well-defined flow $F_t$. Let us identify $TQ \times \mathbb{R}  $ with the space of solution of \eqref{simple_system} by using the correspondence
\[
(q_0, v_0, S_0) \in TQ  \times\mathbb{R}  \longleftrightarrow F_t(q_0, v_0, S_0)=(q(t), v(t), S(t)) \in \text{Solutions of \eqref{simple_system}} ,
\]
where $F_t: TQ  \times\mathbb{R} \to TQ  \times\mathbb{R}$ is the flow of the system \eqref{simple_system}. 
\medskip

We define the horizontal one-forms $ \omega ^{\rm fr}, \omega ^{\rm ext}\in \Omega ^1 (TQ \times \mathbb{R}  )$ associated to the friction and external forces by
\begin{align*} 
\omega ^{\rm fr}(q,v,S)\cdot(\delta q,\delta v,\delta S)&:= \left<F^{\rm fr}(q,v,S), \delta q\right>,\\[1mm]
\omega ^{\rm ext}(q,v,S)\cdot(\delta q,\delta v,\delta S)&:= \left<F^{\rm ext}(q,v,S), \delta q\right>,
\end{align*}
where $(\delta q,\delta v,\delta S) \in T_{(q,v,S)}(TQ \times \mathbb{R})$.
We also define the one-form $ \omega^{ \tau  }:=  T \mbox{d}S$ on $TQ \times \mathbb{R}  $ by
\[
\omega{^\tau }(q,v,S)\cdot( \delta q, \delta v, \delta S):= T(q,S)\delta S = - \frac{\partial L}{\partial S}(q,v,S) \delta S.
\]
In order to derive the structure preserving property, we shall extend the argument used in \cite[\S 1.2.3]{MaWe2001}.
Let us define the restricted action map as
\[
\hat{ \mathfrak{S} }: TQ \times \mathbb{R}   \rightarrow \mathbb{R}  , \quad \hat {\mathfrak{S} }( q_0 , v _0 , S_0):=\int_0^T L\big(F_t(q_0,v_0, S_0)\big) dt.
\]
The derivative of this map reads
 \fontsize{9pt}{13pt}\selectfont
\begin{align*}
&\mathbf{d} \hat { \mathfrak{S} }( q_0 , v_0, S_0) \cdot (\delta q_0, \delta v_0, \delta S_0)\\
&= \int_0^T \left< \frac{\partial L}{\partial q}(q(t), \dot q(t), S(t))- \frac{d}{dt}\frac{\partial L}{\partial \dot q}(q(t), \dot q(t), S(t)), \delta q (t) \right>dt + 
\left<\frac{\partial L}{\partial \dot q}(q(t), \dot q(t), S(t), \delta q (t) \right> \bigg|_{t=0}^{t=T} \\
& \qquad \qquad  +\int_0^T  \frac{\partial L}{\partial S}(q(t), \dot q(t), S(t)) \delta S(t) dt\\
&= - \int_0^T \left<F^{\rm fr+ext}(q(t), \dot q(t), S(t)), \delta q (t) \right> \, dt +  \Theta _L(q(t), \dot q(t), S(t))\cdot (\delta q (t),\delta v(t), \delta S(t)) \bigg|_{t=0}^{t=T}\\
& \qquad \qquad  -\int_0^T T(q(t), S(t)) \delta S(t)dt\\
&= - \int_0^TF_t ^\ast \omega ^{\rm fr+ext} ( q_0 , v_0, S_0) \cdot (\delta q_0, \delta v_0,\delta S_0)\, dt  \\
&\qquad\qquad  + \left( F_T ^\ast \Theta_L-  \Theta _L  \right) \cdot (q_0, v_0, S_0) ( \delta q_0, \delta v_0, \delta S_0) -\int_0^T F_t ^\ast \omega ^{ \tau } (q_0, v_0, S_0) \cdot ( \delta q_0, \delta v_0, \delta S_0)dt,
\end{align*}
\normalsize 
\noindent
where we used the notations $F^{\rm fr+ext}:=F^{\rm fr} + F^{\rm ext}$ and  $ \omega  ^{\rm fr+ext}:  =\omega  ^{\rm fr} + \omega  ^{\rm ext}$.
Thus we obtain the relation
\[
\mathbf{d} \hat {\mathfrak{S} }=F_T ^\ast \Theta _L - \Theta _L - \int_0^T F_t ^\ast ( \omega  ^{\rm fr+ext+ \tau } )d t. 
\]
as one-forms on $TQ \times \mathbb{R} $, where $ \omega  ^{\rm fr+ext+ \tau } :=\omega  ^{\rm fr} + \omega  ^{\rm ext}+ \omega  ^{\rm \tau } $.
By taking the exterior derivative of this equality, we obtain the following result.

\begin{theorem}\label{theorem_continuous}  Consider a simple thermodynamic system and assume that the Lagrangian $L(q,\dot{q},S)$ is regular and the physical assumptions \eqref{assumption} and \eqref{physical_assumption} are verified. Then \eqref{simple_system} defines a well-defined flow $F_t$ on $TQ  \times\mathbb{R} $. This flow verifies the following generalization of the symplectic property \eqref{symplecticity} of the flow in classical mechanics:
\begin{equation}\label{symplecticity_thermo} 
F_T^\ast  \Omega _L= \Omega _L-\mathbf{d} \int_0^T F_t ^\ast ( \omega ^{\rm fr+ext+ \tau })\, dt.
\end{equation} 
\end{theorem} 

\medskip

Note that we can write this property as
\[
F_T^\ast  \Omega _L= \Omega _L-\int_0^T F_t ^\ast ( \mathbf{d} \omega ^{\rm fr+ext+ \tau })\, dt.
\]
%
%
%

\subsection{Thermodynamics of simple systems - discrete case}\label{DiscreteSimpThermo}

In this section, we will show that the discrete flow of our variational integrator satisfies a discrete analogue of the property \eqref{symplecticity_thermo}. We assume that the discrete thermodynamical system satisfies the regularity criteria \eqref{regularity_criteria}. This ensures the existence of the discrete flow $F_{L_d}: C_K ^d \rightarrow C_K ^d$:
\[
(q _k , q_{k+1}, S_k, S_{k+1})  \in C_K ^d \mapsto (q _{k+1} , q_{k+2}, S_{k+1}, S_{k+2})  \in C_K ^d,
\]
obtained by solving the numerical scheme \eqref{discrete_GBYo}, namely,
\begin{equation}\label{discrete_GBYo_bis}
\left\{
\begin{array}{l}
 D_1L_d(q _k , q_{k+1}, S_k, S_{k+1}) + D_2L_d(q _{k-1} , q_k, S_{k-1}, S_k)\\
\qquad \qquad \qquad \qquad \quad + (F^{\rm fr -}+ F^{\rm ext -}) ( q _k, q_{k+1}, S _k , S_{k+1})\\
\qquad \qquad \qquad \qquad \quad + (F^{\rm fr +}+  F^{\rm ext +})( q _{k-1}, q_k, S _{k-1} , S_k)=0,\\[2mm]
(q _k, q_{k+1}, S _k , S_{k+1}) \in C_K ^d.
\end{array}
\right.
\end{equation}
In order to formulate the property of the discrete flow, we need to define the following discrete forms on $(Q \times Q) \times (\mathbb{R}  \times \mathbb{R}  )$.

\begin{definition} Given a discrete thermodynamical system with discrete Lagrangian $L_d$ and discrete friction and external forces $F^{\rm fr \pm}$ and $F^{\rm ext \pm}$, we define the discrete one-forms
\begin{align*} 
\Theta _{L_d, F_d}^-(q_0, q_1, S_0, S_1)&:=-D_1L_d(q_0, q_1, S_0, S_1) \mbox{d}q _0-F_d^-(q_0, q_1, S_0, S_1) \mbox{d} q _0,\\
\Theta _{L_d, F_d}^+(q_0, q_1, S_0, S_1)&:=D_2L_d(q_0, q_1, S_0, S_1)  \mbox{d} q _1+F_d^+(q_0, q_1, S_0, S_1)  \mbox{d}q _1,
\end{align*}
where $F_d^\pm:= F^{\rm fr \pm}+ F^{\rm ext \pm}$, and the discrete one-forms
\begin{align*} 
\omega ^{\rm fr}_d(q _0 , q_1 , S_0 , S _1 ):&= F^{\rm fr -} (q _0 , q_1 , S_0 , S _1 )\mbox{d}q_0+ F^{\rm fr +} (q _0 , q_1 , S_0 , S _1 )\mbox{d}q_1,\\
\omega ^{\rm ext}_d(q _0 , q_1 , S_0 , S _1 ):&= F^{\rm ext -} (q _0 , q_1 , S_0 , S _1 )\mbox{d}q_0+ F^{\rm ext +} (q _0 , q_1 , S_0 , S _1 )\mbox{d}q_1,\\
\omega ^ \tau _d(q _0 , q_1 , S_0 , S _1 ):&= - D_3L_d (q _0 , q_1 , S_0 , S _1 )\mbox{d}S_0 -  D_4L_d (q _0 , q_1 , S_0 , S _1 )\mbox{d}S_1,
\end{align*} 
which are the discrete analogue of the one-forms $ \omega ^{\rm fr}$, $\omega ^{\rm ext}$, $\omega^ \tau $ defined in the continuous case earlier.
\end{definition}

 The one-forms $\Theta _{L_d,F_d}^\pm \in \Omega ^1 \big((Q \times Q  )\times ( \mathbb{R}   \times \mathbb{R}  )\big)$ are related to the canonical one-form $ \Theta _{\rm can} \in \Omega ^1  (T^*Q)$ as
\[
\Theta _{L_d,F_d}^\pm= ( \mathbb{F}  L_{F_d}^\pm) ^\ast \Theta _{\rm can},
\]
where $\mathbb{F}  L_{F_d}^\pm: (Q \times Q  )\times ( \mathbb{R}   \times \mathbb{R}  ) \rightarrow T^*Q$ are the discrete Legendre transforms with force defined by
\begin{align*} 
\mathbb{F}  L_{F_d}^-(q_0, q_1, S_0, S_1)&:= \big(q_0, -D_1L_d(q_0,q_1, S_0, S_1)-F_d^-(q_0,q_1, S_0, S_1)\big)=(q_0,p_0),\\
\mathbb{F}  L_{F_d}^+(q_0, q_1, S_0, S_1)&:= \big(q_1, D_2L_d(q_0,q_1, S_0, S_1)+F_d^+(q_0,q_1, S_0, S_1)\big)=(q_1,p_1).
\end{align*}
These one-forms are the natural extensions of the one-forms for the discrete Euler-Lagrange equations with external forces considered in \cite{MaWe2001}.

\medskip

We show below that the discrete flow \eqref{discrete_flow} satisfies a discrete analogue of the property \eqref{symplecticity_thermo} of the continuous flow obtained in Theorem \ref{theorem_continuous}.
To obtain this result, we extend the argument used in \cite[\S 1.3.2]{MaWe2001}. Similarly with the continuous case earlier, we identify the space of solutions of \eqref{discrete_GBYo} with the space of initial conditions $(q _0 , q_1 , S_0 , S _1 )$ and we define the restricted discrete action map
\[
\hat {\mathfrak{S} }_d(q _0 , q_1 , S_0 , S _1 ):= \sum_{k=0} ^{N-1}L _d ( q _k , q_{k+1}, S _k , S_{k+1}),
\]
where on the right hand side, the discrete action functional is evaluated on the solution of \eqref{discrete_GBYo} with initial conditions $(q _0 , q_1 , S_0 , S _1 )$.

\begin{theorem}\label{theorem_discrete} Consider the numerical scheme \eqref{discrete_GBYo_bis} arising from the discrete variational formulation of Definition \ref{discrete_VP} for the nonequilibrium thermodynamic of a simple system. Assume that the regularity criteria \eqref{regularity_criteria} is verified. Then the scheme \eqref{discrete_GBYo_bis} induces a well-defined discrete flow $F_{L_d}: C_K ^d \rightarrow C_K ^d$:
\[
(q _k , q_{k+1}, S_k, S_{k+1})  \in C_K ^d \mapsto (q _{k+1} , q_{k+2}, S_{k+1}, S_{k+2})  \in C_K ^d.
\]
Moreover, this flow verifies the following property
\begin{equation}\label{discrete_symplecticity_thermo}
\big( F_{L_d}^{(N-1)} \big)^\ast \Omega  _{L_d, F_d}^+ - \Omega  _{L_d, F_d}^-= - \mathbf{d} \sum_{k=0} ^{N-1} \big( F_{L_d}^{(k)}\big) ^\ast \omega^ {\rm fr+ ext +\tau} _d,
\end{equation} 
which is a discrete version of the property \eqref{symplecticity_thermo} of the flow of a simple and closed system. This property is also an extension to nonequilibrium thermodynamics of the symplectic property 
\eqref{symplectic_integrator} of the flow of a variational integrator in classical mechanics. 
\end{theorem}
\begin{proof} The first part of the theorem has been proven above. We now prove formula \eqref{discrete_symplecticity_thermo}.
Using the notations
$
L_d ^k :=L _d ( q _k , q_{k+1}, S _k , S_{k+1}) \quad\text{and}\quad F^{k \pm}:= (F^{\rm fr \pm}+ F^{\rm ext \pm}) ( q _k , q_{k+1}, S _k , S_{k+1}),
$
we compute the derivative of $\hat {\mathfrak{S} }_d$ as
\fontsize{9pt}{16pt}\selectfont
\begin{align*} 
&\mathbf{d} \hat {\mathfrak{S} }_d(q _0 , q_1 , S_0 , S _1 ) \cdot (\delta q _0 , \delta q_1 , \delta S_0 , \delta S _1 )\\
&\quad = \sum_{k=0} ^{N-1}D_1L _d ^k \delta q _k  + D_2L _d ^k \delta q _{k+1} +D_3L _d ^k \delta S _k +D_4L _d ^k \delta S _{k+1} \\
&\quad =D_1L _d ^0 \delta q _0+ \sum_{k=1} ^{N-1}( D_1L _d ^k  + D_2L _d ^{k-1}) \delta q _k + D_2L_d^{N-1} \delta q _N
+ \sum_{k=0} ^{N-1} \left( D_3L _d ^k \delta S _k +D_4L _d ^k \delta S _{k+1} \right) \\
&\quad =D_1L _d ^0 \delta q _0- \sum_{k=1} ^{N-1}( F^{k -}_d+F^{k-1 +}_d) \delta q _k + D_2L_d^{N-1} \delta q _N
+ \sum_{k=0} ^{N-1} \left( D_3L _d ^k \delta S _k +D_4L _d ^k \delta S _{k+1} \right) \\
&\quad =D_1L _d ^0 \delta q _0+ F_d^{0 -} \delta q _0 - \sum_{k=0} ^{N-1} \left(  F^{ k -}_d \delta q_k +F^{k +}_d \delta q _{k+1} \right)  \\
&\hspace{3cm}+ F_d^{N-1 +} \delta q_N+ D_2L_d^{N-1} \delta q _N+ \sum_{k=0} ^{N-1} \left( D_3L _d ^k \delta S _k +D_4L _d ^k \delta S _{k+1} \right) \\
&\quad=- \Theta _{L_d, F_d}^-(q_0, q_1, S_0, S_1)( \delta q_0, \delta q_1, \delta S_0, \delta S_1) \\
&\hspace{3cm}-\sum_{k=1} ^{N-1} \omega _d^{\rm fr+ ext}(q_k,q_{k+1}, S_k, S_{k+1}) ( \delta q_k, \delta q_{k+1}, \delta S_k, \delta S_{k+1})\\
&\quad+ \Theta _{L_d, F_d}^+(q_{N-1}, q_N, S_{N-1}, S_N)( \delta q_{N-1}, \delta q_N, \delta S_{N-1}, \delta S_N)
\\
&\hspace{3cm}- \sum_{k=0} ^{N-1}\omega^ \tau _d (q _k , q_{k+1} , S_k , S_{k+1} ) (\delta q _k , \delta q_{k+1} , \delta S_k , \delta S_{k+1} ).
\end{align*} 
\normalsize
By using the notation
\[
F_{L_d}^{(k)}:= \underbrace{F_{L_d} \circ ...\circ F_{L_d}}_{k},
\]
we can write this differential as
\[
\mathbf{d} \hat {\mathfrak{S} }_d=  \big( F_{L_d}^{(N-1)} \big) ^\ast \Theta _{L_d, F_d}^+ -  \Theta _{L_d, F_d}^- - \sum_{k=0} ^{N-1} \big( F_{L_d}^{(k)}\big) ^\ast \omega^ {\rm fr+ ext +\tau} _d .
\]
Taking the exterior derivative of this relation, we have the result.
\end{proof}

\section{Examples}\label{Section_5} 

In this section, we develop several numerical schemes based on the variational integrator for the nonequilibrium thermodynamics derived in Section \ref{Section_3} by considering several standard discretizations of a given Lagrangian. Then, we illustrate our schemes with the example of the mass-spring-friction system moving in an ideal gas. 

\subsection{Variational discretization schemes}

We consider three standard types of approximation of the time integral of a given Lagrangian. This leads to numerical schemes which are extensions  of the {\it Verlet scheme}, of the {\it variational midpoint rule scheme} as well as of the {\it symmetrized Lagrangian variational integrator}. Let us assume $Q= \mathbb{R}  ^n $.

\paragraph{Variational scheme 1.} Let us first choose the finite difference map $\varphi: ( Q \times Q ) \times ( \mathbb{R}  \times\mathbb{R}  ) \to T(Q \times \mathbb{R})$ as
\[
\varphi( q _k , q_{k+1}, S _k , S_{k+1}) =\left(q _k, S_k , \frac{q_{k+1}-q _k }{h} ,\frac{S_{k+1}-S_k}{h} \right) .
\]
For a given Lagrangian $L(q, \dot q, S)$, the discrete Lagrangian in \eqref{construction_L_d} thus reads
\[
L_d( q _k , q_{k+1}, S _k , S_{k+1})= hL\left(q _k, \frac{q_{k+1}-q _k }{h} ,S_k \right)
\]
and the discrete phenomenological constraint \eqref{def_CKd} is given here by
\[
\frac{\partial L}{\partial S}( q_k , S _k ) \frac{S_{k+1}-S_k}{h}= F^{\rm fr} \left( q_k ,\frac{q_{k+1}-q_k}{h},  S _k\right)  \frac{q_{k+1}-q_k}{h}.
\]
The natural discretization of the forces $F^{\rm fr}$ and $F^{\rm ext}$ associated to this discretization of the Lagrangian may be given as follows (see \cite[\S3.2.5]{MaWe2001}):
\begin{align*} 
F^{\rm fr -}_d(q_k, q_{k+1}, S_k, S_{k+1})&= hF^{\rm fr} \left(q _k, \frac{q _{k+1} - q_k}{h}, S_k \right), \quad F^{\rm fr +}_d(q_k, q_{k+1}, S_k, S_{k+1})=0,\\
F^{\rm ext -}_d(q_k, q_{k+1}, S_k, S_{k+1})&= hF^{\rm ext} \left(q _k, \frac{q _{k+1} - q_k}{h}, S_k \right), \quad F^{\rm ext +}_d(q_k, q_{k+1}, S_k, S_{k+1})=0.
\end{align*} 
The first equation in \eqref{discrete_GBYo} thus becomes
\begin{align*} 
&\frac{1}{h} \left[ \frac{\partial L}{\partial v}\left(  q _k , \frac{q _{k+1} - q_k }{h}, S_k\right) -  \frac{\partial L}{\partial v}\left(  q _{k-1}  , \frac{q _k - q_{k-1} }{h}, S_{k-1}\right)\right]- \frac{\partial L}{\partial q}\left( q _k, \frac{q _{k+1} - q_k}{h}, S_k\right)\\[2mm]
& \qquad\qquad =F^{\rm fr} \left(q _k, \frac{q _{k+1} - q_k}{h}, S_k \right)+F^{\rm ext} \left(q _k, \frac{q _{k+1} - q_k}{h}, S_k \right).
\end{align*} 

For the standard Lagrangian
\begin{equation}\label{standard_L} 
L(q,v, S)= \frac{1}{2} m |v|^2 -U(q,S), 
\end{equation} 
where $v= \dot{q}$, we obtain the following numerical scheme:
\begin{framed}
\paragraph{Scheme 1:}
{\fontsize{9pt}{18pt}\selectfont
\[
\left\{ 
\begin{array}{l}
\vspace{0.2cm}  m \displaystyle\frac{q_{k+1}- 2q_k+q_{k-1}}{h ^2 } +\frac{\partial U}{\partial q}(q_k, S_k)= F^{\rm fr} \left(q _k, \frac{q _{k+1} - q_k}{h}, S_k \right)+F^{\rm ext} \left(q _k, \frac{q _{k+1} - q_k}{h}, S_k \right),\\[5mm]
\displaystyle\frac{\partial U}{\partial S}(q _k , S_k ) \frac{S_{k+1}-S_k}{h}= -F^{\rm fr} \left(q _k, \frac{q _{k+1} - q_k}{h}, S_k \right) \frac{q_{k+1}-q_k}{h}.
\end{array}
\right.
\]}
\end{framed}
\noindent This is an {\bfi extension of the Verlet scheme} to nonequilibrium thermodynamics. The matrix \eqref{regularity_criteria} for Scheme 1 has the entries:
\begin{equation}\label{Coeff_matrix1}
\left\{ 
\begin{array}{ll}
A _{11}= - \displaystyle \frac{m}{h}+ \frac{\partial F}{\partial v}, \quad F:=F^{\rm fr}+ F^{\rm ext},\;\;& A_{12}=0,\\[5mm]
A_{21}=  \displaystyle\frac{1}{h} \frac{\partial F^{\rm fr}}{\partial v} \frac{q _1 - q _0 }{h} +\frac{1}{h} F^{\rm fr},\;\; &
A_{22}=  \displaystyle\frac{1}{h}\frac{\partial U}{\partial S},
\end{array}
\right.
\end{equation}
where $F^{\rm ext}=F^{\rm ext}\left( q_{0}, \frac{q_{1}-q_{0}}{h}, S_{0}\right)$, $F^{\rm fr}=F^{\rm fr}\left(q_{0}, \frac{q_{1}-q_{0}}{h}, S_{0}\right)$ and $U=U(q_{0},S_{0})$. The regularity criteria \eqref{regularity_criteria} is thus satisfied if and only if
\[
\frac{\partial U}{\partial S}(q_0,S_0) \frac{1}{h}\neq 0\quad\text{and}\quad - \frac{m}{h}+\frac{\partial F^{\rm fr}}{\partial v} \left( q _0 ,\frac{q_1-q_0}{h}, S_0 \right)  \neq 0.
\]
The first condition is always satisfied under the physical assumption \eqref{physical_assumption}. The second condition is satisfied for all friction forces that are linear in velocity.

\paragraph{Variational scheme 2.} 
More generally, we can choose a finite difference map of the form
\[
\varphi( q _k , q_{k+1}, S _k , S_{k+1}) =\left( (1- \alpha ) q _k+  \alpha q_{k+1} , \frac{q_{k+1}-q _k }{h}, (1- \alpha ) S _k+  \alpha S_{k+1},\frac{S_{k+1}-S_k}{h} \right)
\]
for some parameter $ \alpha \in [0,1]$.
For $ \alpha = \frac{1}{2} $, we have
\[
L _d( q _k , q_{k+1}, S _k , S_{k+1}):= h L\left( \frac{q _k +  q _{k+1}}{2} , \frac{q _{k+1} - q_k }{h}, \frac{S _k +  S _{k+1}}{2}\right).
\]
The natural discretization of the force $F^{\rm fr}$ associated to this discretization of the Lagrangian is (see \cite[\S3.2.5]{MaWe2001})
\begin{equation*}\label{forces_2}
\begin{aligned} 
F^{\rm fr -}_d(q_k, q_{k+1}, S_k, S_{k+1})&= h \frac{1}{2} F^{\rm fr} \left(\frac{q _k+q_{k+1}}{2}, \frac{q _{k+1} - q_k}{h}, \frac{S _k+S_{k+1}}{2} \right)\\[2mm]
&= F^{\rm fr +}_d(q_k, q_{k+1}, S_k, S_{k+1}),
\end{aligned}
\end{equation*}  
similarly for $F^{\rm ext }$. The discrete phenomenological constraint \eqref{def_CKd} is given here by
{\fontsize{9pt}{18pt}\selectfont\[
\frac{\partial L}{\partial S} \left( \frac{q _k +  q _{k+1}}{2} ,  \frac{S _k +  S _{k+1}}{2} \right) \frac{S_{k+1}-S_k}{h}= F^{\rm fr} \left( \frac{q _k +  q _{k+1}}{2} , \frac{q _{k+1} - q_k }{h}, \frac{S _k +  S _{k+1}}{2}\right)  \frac{q_{k+1}-q_k}{h}.
\]}
The first equation in \eqref{discrete_GBYo} is
{\fontsize{9pt}{13pt}\selectfont
\begin{align*} 
&\frac{1}{h} \left[ \frac{\partial L}{\partial v}\left(  \frac{q _k + q _{k+1} }{2} , \frac{q _{k+1} - q_k }{h}, \frac{S _k +  S _{k+1}}{2}\right) -  \frac{\partial L}{\partial v}\left(  \frac{q _{k-1}+ q _k }{2}  , \frac{q _k - q_{k-1} }{h}, \frac{S _{k-1} +  S _k}{2}\right)\right]\\[3mm]
& \qquad -\frac{1}{2} \left[  \frac{\partial L}{\partial q}\left(  \frac{q _{k-1}+ q _k }{2}  , \frac{q _k - q_{k-1} }{h}, \frac{S _{k-1} +  S _k}{2}\right)+\frac{\partial L}{\partial q}\left( \frac{q _k + q _{k+1} }{2} , \frac{q _{k+1} - q_k }{h}, \frac{S _k +  S _{k+1}}{2}\right) \right]\\[3mm]
&=\frac{1}{2} F^{\rm fr} \left(\frac{q _k+q_{k+1}}{2}, \frac{q _{k+1} - q_k}{h}, \frac{S _k+S_{k+1}}{2} \right)+ \frac{1}{2} F^{\rm fr} \left(\frac{q _{k-1}+q_k}{2}, \frac{q _k - q_{k-1}}{h}, \frac{S _{k-1}+S_k}{2} \right)\\
&\qquad+ \frac{1}{2} F^{\rm ext} \left(\frac{q _k+q_{k+1}}{2}, \frac{q _{k+1} - q_k}{h}, \frac{S _k+S_{k+1}}{2} \right)+ \frac{1}{2} F^{\rm ext} \left(\frac{q _{k-1}+q_k}{2}, \frac{q _k - q_{k-1}}{h}, \frac{S _{k-1}+S_k}{2} \right).
\end{align*} }
The corresponding expressions for arbitrary $ \alpha \in [0,1]$ are derived similarly.

For the standard Lagrangian \eqref{standard_L} we obtain the following numerical scheme:
\begin{framed}
\paragraph{Scheme 2:}{\fontsize{9pt}{18pt}\selectfont
\[
\left\{ 
\begin{array}{l}
\vspace{0.2cm} m \displaystyle\frac{q_{k+1}- 2q_k+q_{k-1}}{h ^2 }+ \frac{1}{2}\left[ \frac{\partial U}{\partial q} \left( \frac{q_{k-1}+q_k}{2}, \frac{S_{k-1}+S_k}{2} \right) + \frac{\partial U}{\partial q} \left( \frac{q_k+q_{k+1}}{2}, \frac{S_k+S_{k+1}}{2} \right)  \right]\\[2mm]
\vspace{0.2cm} \quad =\displaystyle\frac{1}{2} F^{\rm fr} \left(\frac{q _k+q_{k+1}}{2}, \frac{q _{k+1} - q_k}{h}, \frac{S _k+S_{k+1}}{2} \right)+ \frac{1}{2} F^{\rm fr} \left(\frac{q _{k-1}+q_k}{2}, \frac{q _k - q_{k-1}}{h}, \frac{S _{k-1}+S_k}{2} \right)\\[2mm]
\vspace{0.2cm}\;\;+ \displaystyle\frac{1}{2} F^{\rm ext} \left(\frac{q _k+q_{k+1}}{2}, \frac{q _{k+1} - q_k}{h}, \frac{S _k+S_{k+1}}{2} \right)+ \frac{1}{2} F^{\rm ext} \left(\frac{q _{k-1}+q_k}{2}, \frac{q _k - q_{k-1}}{h}, \frac{S _{k-1}+S_k}{2} \right),\\[7mm]
\vspace{0.2cm}\displaystyle\frac{1}{2}\left[\frac{\partial U}{\partial S}\left( \frac{q_{k-1}+q_{k}}{2}, \frac{S_{k-1}+S_{k}}{2} \right)+\frac{\partial U}{\partial S}\left( \frac{q_k+q_{k+1}}{2}, \frac{S_k+S_{k+1}}{2} \right)\right]\frac{S_{k+1}-S_k}{h}\\[2mm]
\qquad\qquad= - F^{\rm fr} \left( \displaystyle\frac{q _k +  q _{k+1}}{2} , \frac{q _{k+1} - q_k }{h}, \frac{S _k +  S _{k+1}}{2}\right)  \displaystyle\frac{q_{k+1}-q_k}{h}.
\end{array}
\right.
\]}
\end{framed}
\noindent This is an extension to nonequilibrium thermodynamics of the {\bfi variational midpoint rule scheme}. 
The matrix \eqref{regularity_criteria} has the entries:
\begin{equation}\label{Coeff_matrix2}
\left\{ {\fontsize{9pt}{18pt}\selectfont
\begin{array}{l}
A _{11}= - \displaystyle\frac{m}{h}-\frac{h}{2}\frac{\partial ^2 U}{\partial  q^2 }+ \frac{h}{4} \frac{\partial F}{\partial q}+ \frac{1}{2} \frac{\partial F}{\partial v}, \quad F:=F^{\rm fr}+ F^{\rm ext},\\[5mm]
A_{12}= - \displaystyle\frac{h}{4} \frac{\partial ^2 U}{\partial S\partial q} + \frac{h}{4} \frac{\partial F}{\partial S},\\[5mm]
A_{21}= \displaystyle\frac{1}{2} \frac{\partial ^2 U}{\partial S\partial q}  \frac{S_1-S_0}{h}  + \left(\frac{1}{2} \frac{\partial F^{\rm fr}}{\partial q} + \frac{1}{h} \frac{\partial F^{\rm fr}}{\partial v} \right)\frac{q _1 - q _0 }{h} +\frac{1}{h} F^{\rm fr},\\[5mm]
A_{22}= \displaystyle\frac{1}{2} \frac{\partial ^2 U}{\partial  S^2} \frac{S_1-S_0}{h}  + \frac{1}{h}\frac{\partial U}{\partial S}+ \frac{1}{2} \frac{\partial F^{\rm fr}}{\partial S} \frac{q_1 - q_0}{h},
\end{array}}
\right.
\end{equation}
where 
\[
U=U\left( \frac{q_0+q_{1}}{2}, \frac{S_0+S_{1}}{2} \right)\;\;\textrm{and}\;\; F^{\rm fr}=F^{\rm fr} \left( \frac{q _0 +  q _{1}}{2} , \frac{q _{1} - q_0 }{h}, \frac{S _0 +  S _{1}}{2}\right).
\]

\paragraph{Variational scheme 3.} We can choose to approximate the Lagrangian by the symmetrized discrete Lagrangian as follows
\[
L_d( q _k , q _{k+1}, S_k, S_{k+1} )=\frac{1}{2} h L \left(q _k ,\frac{q_{k+1}-q_k}{h}, S_k \right)+\frac{1}{2} h L \left(q _{k+1} ,\frac{q_{k+1}-q_k}{h}, S_{k+1} \right). 
\]
The associated natural choice of discrete forces is given by
\begin{equation*}\label{forces_3}
\begin{aligned} 
F^{\rm fr -}_d(q_k, q_{k+1}, S_k, S_{k+1})&= h \frac{1}{2}  F^{\rm fr} \left(q_k, \frac{q _{k+1} - q_k}{h}, S_k \right), \\[3mm]
F^{\rm fr +}_d(q_k, q_{k+1}, S_k, S_{k+1})&=h \frac{1}{2} F^{\rm fr} \left(q_{k+1}, \frac{q _{k+1} - q_k}{h}, S_{k+1} \right),
\end{aligned}
\end{equation*}  
similarly for $F^{\rm ext \pm}$; see \cite[p. 29]{KaMaOrWe2000} ($ \alpha =0$). A natural discrete phenomenological constraint is given here by
\begin{align*} 
&\left[ \frac{\partial L}{\partial S}( q_k , S _k ) + \frac{\partial L}{\partial S}( q_{k+1} , S _{k+1} ) \right] \frac{S_{k+1}-S_k}{h}\\
&\qquad\qquad = \left[ F^{\rm fr} \left( q_k ,\frac{q_{k+1}-q_k}{h},  S _k\right)  + F^{\rm fr} \left( q_{k+1} ,\frac{q_{k+1}-q_k}{h},  S _{k+1}\right) \right] \frac{q_{k+1}-q_k}{h}.
\end{align*} 
\normalsize
The first equation in \eqref{discrete_GBYo} is
{\fontsize{9pt}{18pt}\selectfont
\begin{align*} 
&\frac{1}{h}\frac{1}{2} \left[ \frac{\partial L}{\partial v} \left( q _k , \frac{q _{k+1}-q_k}{h} \right) +  \frac{\partial L}{\partial v} \left( q _{k+1} , \frac{q _{k+1}-q_k}{h} \right) - \frac{\partial L}{\partial v} \left( q _{k-1} , \frac{q _k-q_{k-1}}{h} \right) -\frac{\partial L}{\partial v} \left( q _k , \frac{q _k-q_{k-1}}{h} \right) \right]  \\[3mm]
&  \qquad\qquad \qquad - \frac{1}{2}\left[\frac{\partial L}{\partial q} \left( q _k , \frac{q _k-q_{k-1}}{h} \right)+\frac{\partial L}{\partial q} \left( q _k , \frac{q _{k+1}-q_k}{h} \right)  \right]\\[3mm]
&= \frac{1}{2} \left(  F^{\rm fr} \left(q_k, \frac{q _{k+1} - q_k}{h}, S_k \right)+F^{\rm fr} \left(q_k, \frac{q _k - q_{k-1}}{h}, S_k \right)  \right) \\[3mm]
& \qquad\qquad \qquad + \frac{1}{2} \left(  F^{\rm ext} \left(q_k, \frac{q _{k+1} - q_k}{h}, S_k \right)+F^{\rm ext} \left(q_k, \frac{q _k - q_{k-1}}{h}, S_k \right)  \right).
\end{align*}}

For the standard Lagrangian \eqref{standard_L} we obtain the following numerical scheme:
\begin{framed}
\paragraph{Scheme 3:}
{\fontsize{9pt}{18pt}\selectfont
\[
\left\{ 
\begin{array}{l}
\vspace{0.2cm}m \displaystyle\frac{q_{k+1}- 2q_k+q_{k-1}}{h ^2 } +\frac{\partial U}{\partial q}(q_k, S_k)= \frac{1}{2} \left[  F^{\rm fr} \left(q_k, \frac{q _{k+1} - q_k}{h}, S_k \right)+F^{\rm fr} \left(q_k, \frac{q _k - q_{k-1}}{h}, S_k \right)  \right] \\[2mm]
\vspace{0.2cm} \qquad\qquad \qquad + \displaystyle\frac{1}{2} \left[  F^{\rm ext} \left(q_k, \displaystyle\frac{q _{k+1} - q_k}{h}, S_k \right)+F^{\rm ext} \left(q_k, \displaystyle\frac{q _k - q_{k-1}}{h}, S_k \right)  \right],\\[5mm]
\vspace{0.2cm} \left[ \displaystyle\frac{\partial U}{\partial S}( q_k , S _k ) + \frac{\partial U}{\partial S}( q_{k+1} , S _{k+1} ) \right] \displaystyle\frac{S_{k+1}-S_k}{h}\\[2mm]
\vspace{0.2cm} \qquad \qquad \qquad  = -\left[ F^{\rm fr} \left( q_k ,\displaystyle\frac{q_{k+1}-q_k}{h},  S _k\right)  + F^{\rm fr} \left( q_{k+1} ,\displaystyle\frac{q_{k+1}-q_k}{h},  S _{k+1}\right) \right] \displaystyle\frac{q_{k+1}-q_k}{h}.
\end{array}
\right.
\]}
\end{framed}

This is a {\bfi  symmetrized Lagrangian variational integrator} applied to nonequilibrium thermodynamics. The matrix \eqref{regularity_criteria} for Scheme 3 has the entries:
\begin{equation}\label{Coeff_matrix3}
\left\{ {\fontsize{9pt}{18pt}\selectfont
\begin{array}{l}
A _{11}= -  \displaystyle\frac{m}{h}+\frac{1}{2}\frac{\partial F_{0}}{\partial v}, \quad F_{0}:=F_{0}^{\rm fr}+ F_{0}^{\rm ext},\\[5mm]
A_{12}=\displaystyle\frac{1}{2}\frac{\partial F_{0}}{\partial S},\\[5mm]
A_{21}=  \displaystyle\frac{\partial^{2} U_{1}}{\partial S \partial q}\frac{S_{1}-S_{0}}{h}+ \left[\frac{1}{h}\frac{\partial F_{0}^{\rm fr}}{\partial v}
+\frac{1}{h}\frac{\partial F_{1}^{\rm fr}}{\partial v}+\frac{\partial F_{1}^{\rm fr}}{\partial q}
\right]\frac{q_{1}-q_{0}}{h}+\frac{F^{\rm fr}_{0}+F^{\rm fr}_{1}}{h},\\[5mm]
A_{22}=  \displaystyle\frac{\partial^{2} U_{1}}{\partial S^{2}}\frac{S_{1}-S_{0}}{h}+\left(\frac{\partial U_{0}}{\partial S}+\frac{\partial U_{1}}{\partial S}\right)\frac{1}{h}+\frac{\partial F_{1}^{\rm fr}}{\partial S}\frac{q_{1}-q_{0}}{h},
\end{array}}
\right.
\end{equation}
where $F_{0}=F\left( q_{0}, \frac{q_{1}-q_{0}}{h}, S_{0}\right)$,  $F^{\rm fr}_{0}=F^{\rm fr}\left(q_{0}, \frac{q_{1}-q_{0}}{h}, S_{0}\right)$, $F^{\rm fr}_{1}=F^{\rm fr}\left(q_{1}, \frac{q_{1}-q_{0}}{h}, S_{1}\right)$, $U_{0}=U(q_{0},S_{0})$ and $U_{1}=U(q_{1},S_{1})$.

\bigskip


\subsection{Example: a mass-spring-friction system moving in an ideal gas}

We consider the example of a mass-spring-friction system moving in a closed room filled with an ideal gas. This system, denoted $ \boldsymbol{\Sigma} $, is illustrated in Fig.\ref{FrictionMassSpring}. We refer to \cite{FeGr2010} for the derivation of the equations of motion for this system from Stueckelberg's point of view.
We consider this simple example since we can take advantage of the fact that the equations of evolution for this system can be explicitly solved. 
This allows us to easily estimate the numerical validity of the scheme to simulate the entropy, temperature, and internal energy behaviors.
\medskip

\paragraph{Continuous setting.} The Lagrangian $L$ for the system $\boldsymbol{\Sigma}$ is given by $L(x,\dot x, S)=\frac{1}{2} m\dot x ^2 -U(x,S)$, where $(x,\dot{x},S)$ denotes the state of the system,
$m$ is the mass of the solid, and $k$ is the spring constant, and where $U(x, S)=\frac{1}{2} k x ^2+\mathcal{U}(S)$ is the potential energy.
The internal energy of the ideal gas is given by $\mathcal{U}=cN\!RT$, where $c$ is the gas constant, $N$ is the number of moles, $R$ is the universal gas constant, and $T$ is the temperature$^{2}$. Note that $\mathcal{U}$ may be rewritten as a function
\[
\mathcal{U}(S,N,V)= \mathcal{U} _0 e^ {\frac{1}{cR}\left(\frac{S}{N}- \frac{S _0 }{N _0 }\right)}\left( \frac{N}{N_0}\right) ^{ \frac{1}{c}+1} \left( \frac{V _0 }{V}\right) ^{\frac{1}{c}},
\]
where $\mathcal{U} _0$ indicates the initial value of the internal energy, $N_{0}$ is the initial mole number of the ideal gas, and $V$ is the volume of the room with the ideal gas, which is assumed to be constant, i.e., $V=V_{0}$. 
We assume that the friction force is given by a viscous force as $F^{\rm fr}(x, \dot x, S)= - \lambda  \dot x$, where $ \lambda \geq 0$ is the phenomenological coefficient determined experimentally and also that the system $\boldsymbol{\Sigma}$ is subject to an external force $F^{\rm ext}$ exerted from the exterior $\boldsymbol{\Sigma}^{\rm ext}$. We also assume that the system is adiabatically closed, so the power due to heat transfer between the system and the exterior is zero, i.e., $P^{\rm ext}_H=0$ and there is no change in the number of moles of the gas, i.e., $N=N_{0}$.

\addtocounter{footnote}{2}
\footnotetext{For simplicity, we neglect the internal energy of the solid. It is given by $\mathcal{U}_s=3N_sRT$, where $N_s$ is number of mole of the solid, and it can be easily included in our discussion.}

\begin{figure}[h!]
\begin{center}
\includegraphics[scale=0.6]{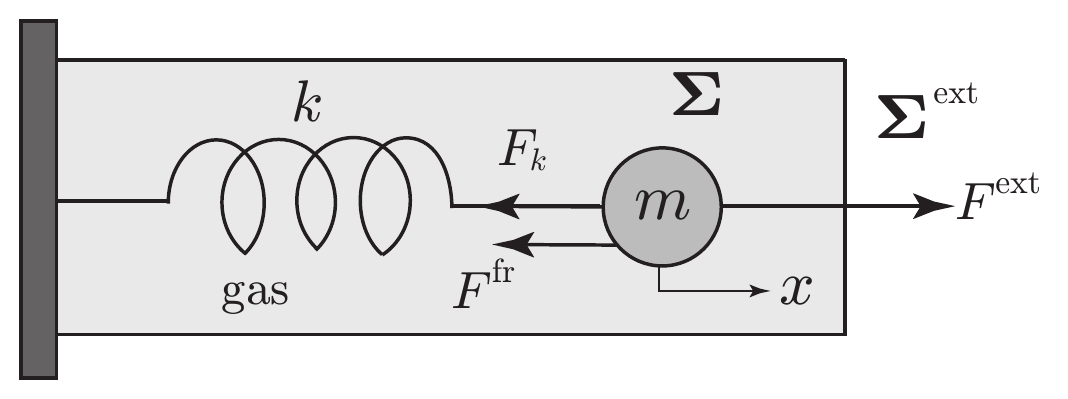}
\caption{ A mass-spring-friction system in a room with gas}
\label{FrictionMassSpring}
\end{center}
\end{figure}
With the above choice of Lagrangian and force, the general equations \eqref{thermo_mech_equations} (arising from the variational formulation \eqref{LdA_thermo}--\eqref{Kinematic_Constraints}) yield the time evolution equations of the coupled mechanical and thermal system as
\[
m\ddot x  =- kx  - \lambda \dot x+F^{\rm ext}(x , \dot x, S), \qquad \dot S= \frac{1}{T} \lambda \dot x^{2},
\]
where 
\[
T=\frac{\partial \mathcal{U}}{\partial S}(S)=\frac{\mathcal{U}_{0}}{cN_{0}R}e^{\frac{1}{cRN_{0}}(S-S_{0})}=T_{0}e^{\frac{1}{cRN_{0}}(S-S_{0})}.
\]
The total energy, given by $E(x, \dot x, S)= \frac{1}{2} m\dot x ^2 +  U(x,S)$, verifies the energy balance equation $ \frac{d}{dt} E= \left<F^{\rm ext}(x , \dot x, S), \dot{x}\right>$.

\paragraph{Exact solutions.}
Consider the special case in which there is no external force, i.e., $F^{\rm ext} =0$. In this case the time evolution equations are given by
\begin{equation}\label{equ_FeGr} 
m\ddot x=- kx- \lambda \dot x  \quad\text{and}\quad T\dot S= \lambda \dot x ^2,
\end{equation}
and the total energy is preserved. These equations can be easily solved explicitly, see \cite{FeGr2010}.
\medskip

Setting $x_{0}=x(0)$ and $v_{0}=\dot{x}(0)$, the solution of the first equation in \eqref{equ_FeGr} is
\begin{equation}\label{MSF_exact_x} 
x(t)= e^{- \kappa t} \left(x_0\cos( \omega t)+ \frac{v_0+\kappa x_0}{ \omega } \sin ( \omega t)\right),
\end{equation}
where 
$$
\kappa= \frac{ \lambda }{2m} ,\;\; \omega _0= \sqrt{ \frac{k}{m} },\quad \omega =\sqrt{ \omega _0 ^2 -\kappa ^2 }, \quad\kappa < \omega _0.
$$
In order to solve the second equation in \eqref{equ_FeGr}, we note that
\begin{equation}\label{note} 
\frac{d}{dt} \mathcal{U}= T \dot S= \lambda \dot x^2  \quad\text{and}\quad \frac{d}{dt} \mathcal{U}= cNR \dot T,
\end{equation} 
from which we obtain the evolution of the temperature as
\begin{equation}\label{MSF_exact_T}
T(t)=T(0)+ \frac{1}{cNR} f(t), \quad \textrm{where
$f(t):=  \lambda \int_0^t \dot x^2(s)ds.$}
\end{equation}
In the above,
\[
\dot x(t)= e^{- \kappa t} \left(v_0\cos( \omega t)-  \left( \kappa \frac{v_0+\kappa x_0}{ \omega }+ x_0 \omega \right)  \sin ( \omega t)\right)
\]
and hence
\begin{align*} 
f(t)&=\lambda \int_0^t \dot x^2(s)\\[2mm]
&=\left(  \frac{1}{2} m v_0 ^2 + \frac{1}{2} k x_0 ^2  \right) - \frac{1}{2(4km- \lambda ^2 )}e^{ -\frac{ \lambda }{m}t} \Big( 4km (mv_0 ^2 + \lambda v_0x_0+kx_0^2 ) \\[2mm]
&\qquad - \lambda (v_0 ^2 \lambda m+ 4 v_0 m k x_0+ \lambda k x_0 ^2 ) \cos (2 \omega t)- \lambda (mv_0 ^2 -k x_0 ^2 )( 4km - \lambda ^2 )^{1/2} \sin( 2 \omega t) \Big).
\end{align*} 
Using \eqref{MSF_exact_T} and  \eqref{note}, we get the explicit evolution of the entropy as
 \begin{equation}\label{MSF_exact_S}
S(t)= S(0)+ cNR \ln \left(\frac{T(t)}{T(0)} \right). 
\end{equation}

\paragraph{Variational discretizations.} We now apply our variational discretization schemes 1, 2, 3  to this example.
\paragraph{\bf \underline{Scheme 1:}} The first scheme yields
\[
\left\{ 
\begin{array}{l}
\vspace{0.2cm}\displaystyle m \frac{x_{k+1}- 2x_k+x_{k-1}}{h ^2 } +kx_k= -\lambda  \frac{x _{k+1} - x_k}{h},\\
\displaystyle\frac{\partial \mathcal{U}}{\partial S}(S_k ) \frac{S_{k+1}-S_k}{h}= \lambda \left(\frac{x _{k+1} - x_k}{h}\right)^{2}.
\end{array}
\right.
\]
The second equation can be restated as
\[
S_{k+1}=\frac{h\lambda}{T_k}\left(\displaystyle\frac{x_{k+1}-x_{k}}{h}\right)^{2}+S_{k},\quad 
\]
where
\[
T_{k}=\frac{\partial \mathcal{U}}{\partial S}(S_{k})=\frac{\mathcal{U}_{0}}{cN_{0}R}e^{\frac{1}{cRN_{0}}(S_{k}-S_{0})}=T_{0}e^{\frac{1}{cRN_{0}}(S_{k}-S_{0})}.
\]

For this example, the matrix \eqref{Coeff_matrix1} of the variational discretization scheme 1 reads
\begin{equation*}\label{regularity_criteria_matrix1} 
\left[
\begin{array}{cc}
-\displaystyle\frac{m}{h}-\lambda & 0 \\[3mm]
0 & \displaystyle\frac{1}{h}T(x_{0},S_{0}) 
\end{array}
\right].
\end{equation*} 
Thus, since $m>0$, $\lambda \geq 0$, and $T>0$,  the discrete flow of the extended Verlet scheme is well-defined.

\paragraph{\underline{Scheme 2:}} The second scheme yields
{\fontsize{9pt}{18pt}\selectfont
\[
\left\{ 
\begin{array}{l}
\vspace{0.2cm}\displaystyle m \frac{x_{k+1}- 2x_k+x_{k-1}}{h ^2 } +\frac{1}{2}k\left(\frac{x_{k-1}+x_{k}}{2}+\frac{x_{k}+x_{k+1}}{2} \right)
\displaystyle = -\frac{1}{2} \lambda\left( \frac{x_{k+1} - x_k}{h}+\frac{x _{k} - x_{k-1}}{h}\right),\\[5mm]
\displaystyle
{\frac{1}{2}}\left[\frac{\partial \mathcal{U}}{\partial S}\left( \frac{S_{k-1}+S_{k}}{2} \right)+\frac{\partial \mathcal{U}}{\partial S}\left( \displaystyle\frac{S_k+S_{k+1}}{2} \right)\right]
\frac{S_{k+1}-S_k}{h}= \lambda \left(\frac{x_{k+1} - x_k}{h}\right)^{2}.
\end{array}
\right.
\]}

For this example, the matrix \eqref{Coeff_matrix2} of the variational discretization scheme 2 reads
\begin{equation*}\label{regularity_criteria_matrix2} 
\left[
\begin{array}{ll}
-\displaystyle\frac{m}{h}-\frac{\lambda}{2} & 0 \\[3mm]
-\displaystyle\frac{2\lambda}{h^{2}}\left(x_{1}-x_{0}\right) &\displaystyle\frac{T}{h}\left(\frac{S_{1}-S_{0}}{2cN_{0}R}+1\right)
\end{array}
\right],
\end{equation*}
where $T=\frac{\partial  \mathcal{U}}{\partial S}\left(\frac{S_{0}+S_{1}}{2} \right)$. 
Thus, since $m>0$, $\lambda \geq 0$, $T>0$, and $cN_{0}R>0$,  the discrete flow of the variational midpoint rule scheme is well-defined.

\vspace{-1cm}

\paragraph{\underline{Scheme 3:}} The third example yields
\[
\left\{ 
\begin{array}{l}
\vspace{0.2cm}\displaystyle m \frac{x_{k+1}- 2x_k+x_{k-1}}{h ^2 } +kx_{k}= -\frac{1}{2} \lambda\left( \frac{x _{k+1} - x_k}{h}+\frac{x _{k} - x_{k-1}}{h}\right),\\[2mm]
\displaystyle \frac{1}{2}\left[\frac{\partial  \mathcal{U}}{\partial S}(S_{k})+ \frac{\partial  \mathcal{U}}{\partial S}(S_{k+1})\right]\frac{S_{k+1}-S_k}{h}= \lambda \left(\frac{x_{k+1} - x_k}{h}\right)^{2}.
\end{array}
\right.
\]

For this example, the matrix \eqref{Coeff_matrix3} of the variational discretization scheme 3 reads
\begin{equation*}\label{regularity_criteria_matrix3} 
\left[
\begin{array}{ll}
-\displaystyle \frac{m}{h}-\frac{\lambda}{2} & 0 \\[5mm]
-\displaystyle \frac{4\lambda}{h^{2}}\left(x_{1}-x_{0}\right) & \displaystyle \frac{T_{1}}{cN_{0}R}\left(\displaystyle \frac{S_{1}-S_{0}}{h}\right)+(T_{0}+T_{1})\displaystyle \frac{1}{h} 
\end{array}
\right],
\end{equation*}
where $T_{0}= \frac{\partial \mathcal{U}}{\partial S}(S_{0})$ and $T_{1}= \frac{\partial \mathcal{U}}{\partial S}(S_{1})$.
Thus, since $m>0$, $\lambda \geq 0$, $T_1>0$, $T_2>0$ and $cN_{0}R>0$,  the discrete flow of the symmetrized Lagrangian variational integrator is well-defined.

\subsection{Numerical tests}

We illustrate the behavior of the three variational schemes for the mass-spring-friction system, by considering two cases of physical parameters and various values for the friction coefficient, namely $ \lambda =0$, $0.2$, $5$, and $10$ ${\rm[N\!\cdot\! s/m]}$.

For each of the five values of $ \lambda $ we display the evolutions of the position, entropy, total energy $E( q _k , q_{k+1}, S _k )$, relative energy errors $\bigr| \frac{ E( q _k , q_{k+1}, S _k ) - E_{0} }{E_{0}}\bigr|$, internal energy, and temperature. Each figure shows the results for the three schemes as well as the exact solution, through $10^{5}$ time steps.

\paragraph{Case 1.} For the first series of numerical tests, we choose the time step $h=10^{-3}{\rm [s]}$ and we set the parameters of the system $\boldsymbol{\Sigma}$ as follows: $m=5\,  {\rm[kg]}$, $N=1\,  {\rm [mol]}$,  $k= 5\, {\rm[N/m]}$, $V=2.494 \times 10^{-2}\,  {\rm [m^{3}]}$. The initial conditions are $x_{0}=0.3\, {\rm [m]}$, $x_{1}=0.3\, {\rm [m]}$, $T_{0}=300\, {\rm [K]}$, $S_{0}=0\, {\rm [J/K]}$.
\medskip

The results for $ \lambda =0$, see Figures \ref{graph_position_entropy_Case1Lamda0} and \ref{graph_energy_errors_Case1Lamda0}, consistently recover the behavior obtained through a usual variational discretization of the Euler-Lagrange equations for the conservative mass-spring system in classical mechanics. In particular, for each scheme the internal energy $\mathcal{U}(S_{k})=\mathcal{U}_{0}$ is preserved and the temperature, given by $T_{k}=\frac{\partial \mathcal{U}}{\partial S}(S_{k})$,  remains a constant, see Figure \ref{graph_interenergy_temp_Case1Lamda0}. Exactly as in the continuous case, in absence of friction in an isolated simple system, the entropy and temperature stay constant, the system is reversible,  and the dynamics is completely described by the Euler-Lagrange equations.

For all the cases with friction, $ \lambda =0.2$, $5$, $10$, the numerical solutions of the position, entropy, internal energy, and temperature reproduce the correct behaviors for all the three schemes, as we see from a direct comparison with the exact solutions, see Figures \ref{graph_position_entropy_Case1Lamda0.2}, \ref{graph_interenergy_temp_Case1Lamda0.2}, \ref{graph_position_entropy_Case1Lamda5}, \ref{graph_interenergy_temp_Case1Lamda5}, \ref{graph_position_entropy_Case1Lamda10}, \ref{graph_interenergy_temp_Case1Lamda10}.

For Scheme 1, the relative energy error is bounded by $10^{-8}$ for all values of $ \lambda $, and decreases in time, whereas for Scheme 2 and 3, the relative energy error is bounded by $10^{-11}$ for all values of $ \lambda $, see Figures \ref{graph_energy_errors_Case1Lamda0.2}, \ref{graph_energy_errors_Case1Lamda5}, \ref{graph_energy_errors_Case1Lamda10}, and stays constant in time.

\pagebreak

\begin{figure}[htbp]
 \vspace{-4cm} \begin{center}
    \begin{tabular}{c}
      \begin{minipage}{0.4\hsize}
        \begin{center}
          \includegraphics[clip, width=5cm]{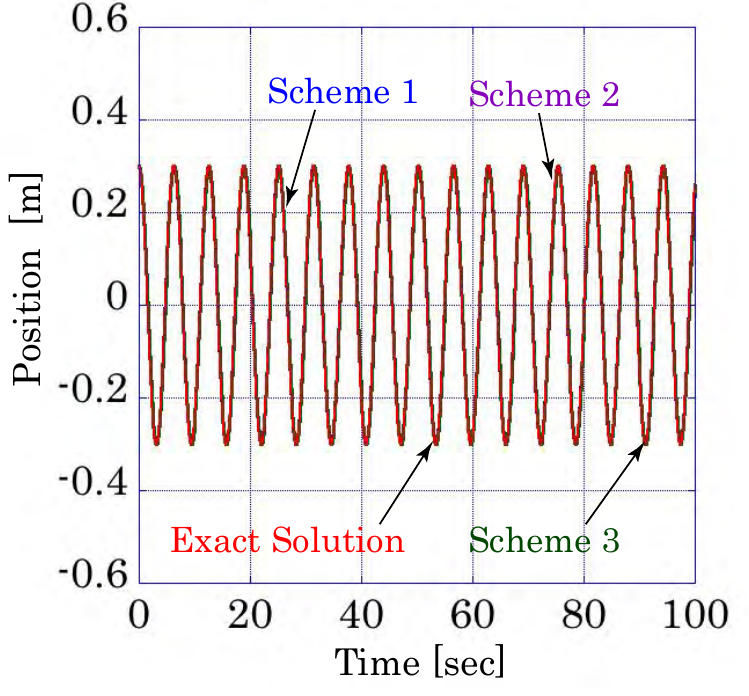}
        \end{center}
      \end{minipage}
      \qquad
      \begin{minipage}{0.4\hsize}
        \vspace{-0.2cm}\begin{center}
          \includegraphics[clip, width=5.2cm]{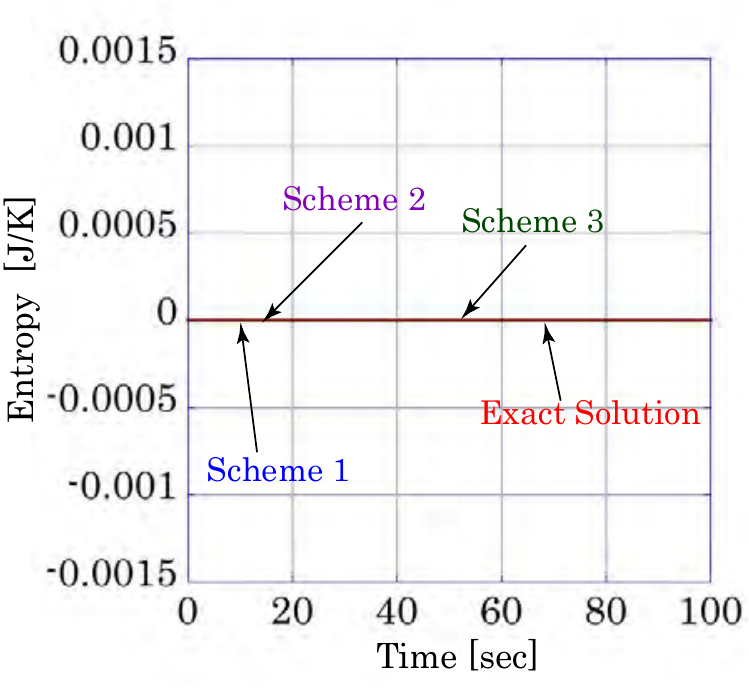}
        \end{center}
      \end{minipage}
    \end{tabular}
   \caption{Time evolutions of position and entropy (Case 1: $\lambda=0$)}
    \label{graph_position_entropy_Case1Lamda0}
  \end{center}
\end{figure}
\begin{figure}[htbp]
  \begin{center}
    \begin{tabular}{c}
      \begin{minipage}{0.4\hsize}
        \begin{center}
          \includegraphics[clip, width=5.4cm]{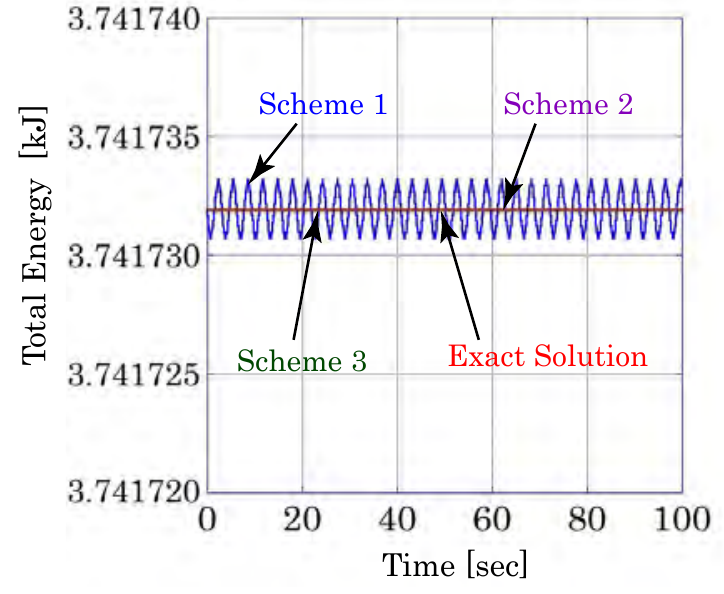}
        \end{center}
      \end{minipage}
      \qquad
      \begin{minipage}{0.4\hsize}
        \vspace{-0.1cm}\begin{center}
          \includegraphics[clip, width=5cm]{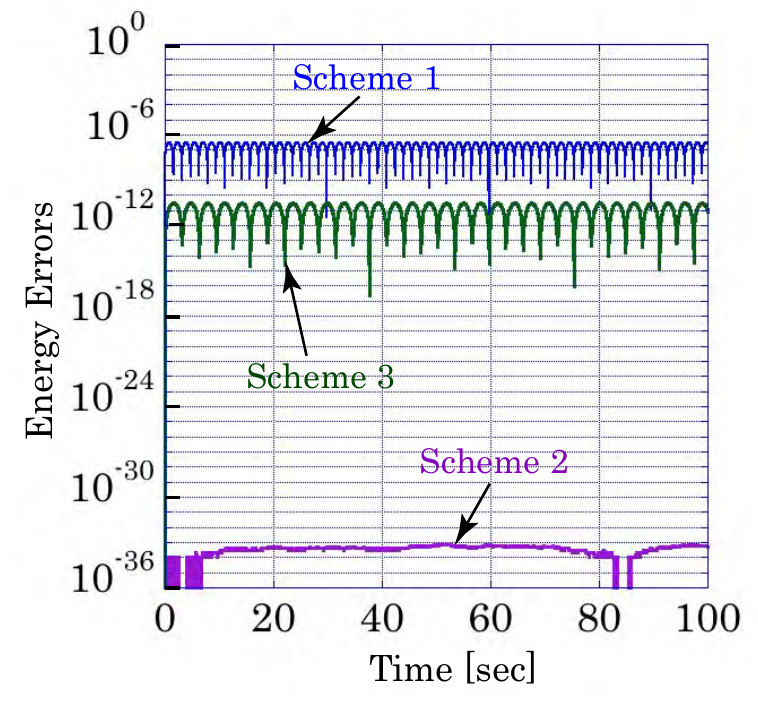}
        \end{center}
      \end{minipage}
    \end{tabular}
   \caption{Total energy and relative energy error (Case 1: $\lambda=0$)}
    \label{graph_energy_errors_Case1Lamda0}
  \end{center}
\end{figure}
\begin{figure}[htbp]
  \begin{center}
    \begin{tabular}{c}
      \begin{minipage}{0.4\hsize}
        \vspace{-0.1cm}\begin{center}
          \includegraphics[clip, width=5.4cm]{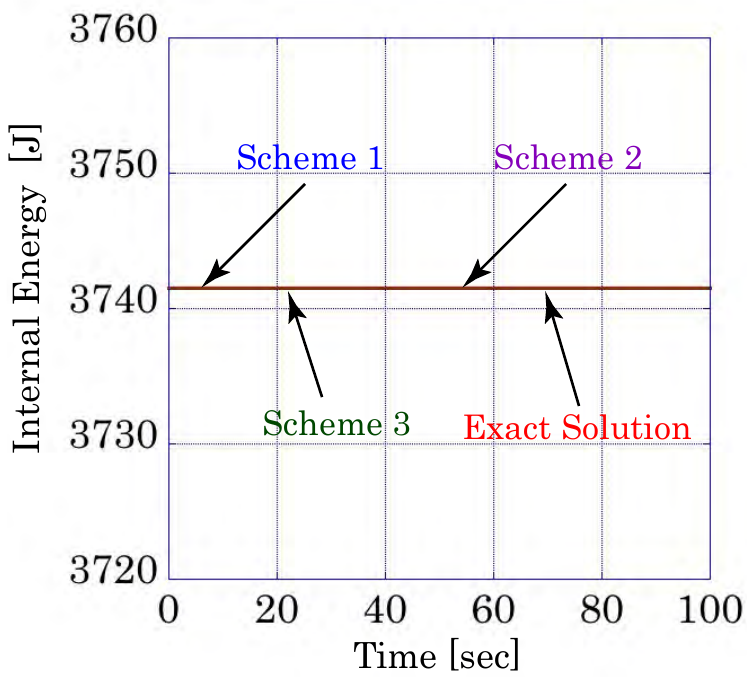}
        \end{center}
      \end{minipage}
      \qquad
      \begin{minipage}{0.4\hsize}
        \begin{center}
          \includegraphics[clip, width=5.4cm]{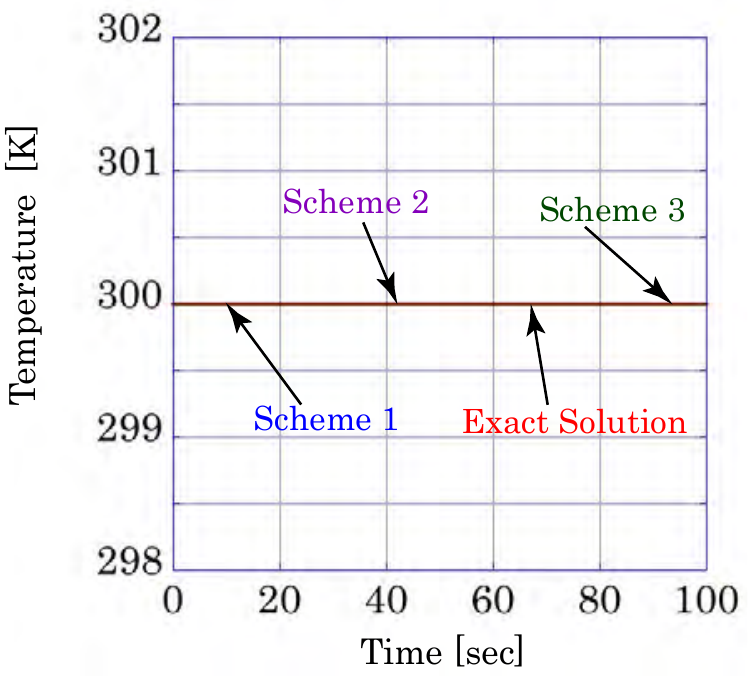}
        \end{center}
      \end{minipage}
    \end{tabular}
   \caption{Internal energy and temperature (Case 1: $\lambda=0$)}
    \label{graph_interenergy_temp_Case1Lamda0}
  \end{center}
\end{figure}



 \begin{figure}[htbp]
  \begin{center}
    \begin{tabular}{c}

      \begin{minipage}{0.4\hsize}
        \begin{center}
          \includegraphics[clip, width=4.9cm]{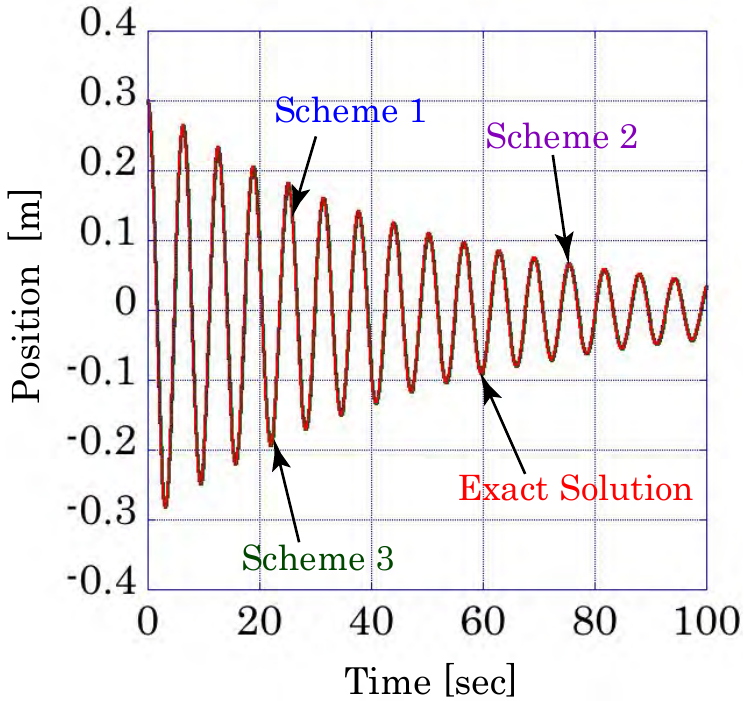}
        \end{center}
      \end{minipage}
      \qquad
      \begin{minipage}{0.4\hsize}
        \begin{center}
          \includegraphics[clip, width=5.4cm]{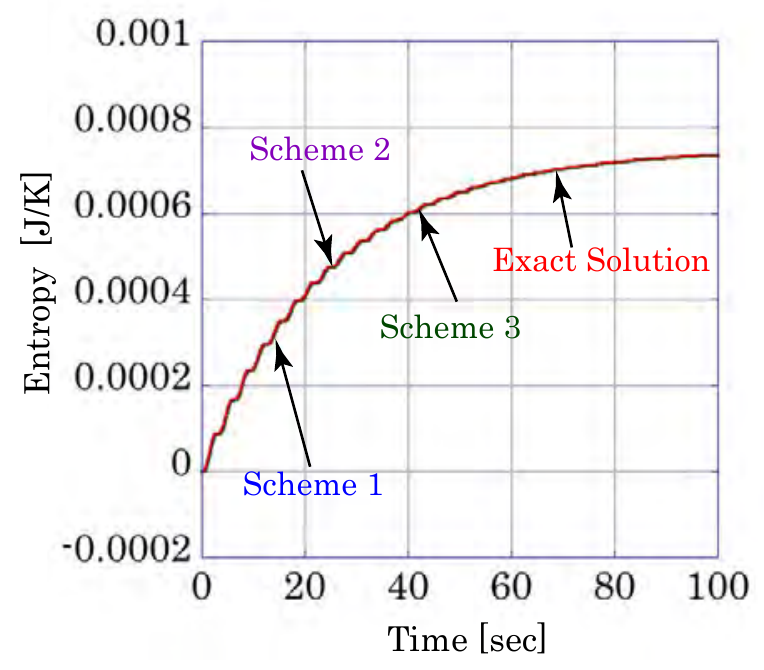}
        \end{center}
      \end{minipage}
    \end{tabular}
   \caption{Time evolutions of position and entropy (Case 1: $\lambda=0.2$)}
    \label{graph_position_entropy_Case1Lamda0.2}
  \end{center}
\end{figure}

\begin{figure}[htbp]
  \begin{center}
    \begin{tabular}{c}

      \begin{minipage}{0.4\hsize}
        \begin{center}
          \includegraphics[clip, width=5.4cm]{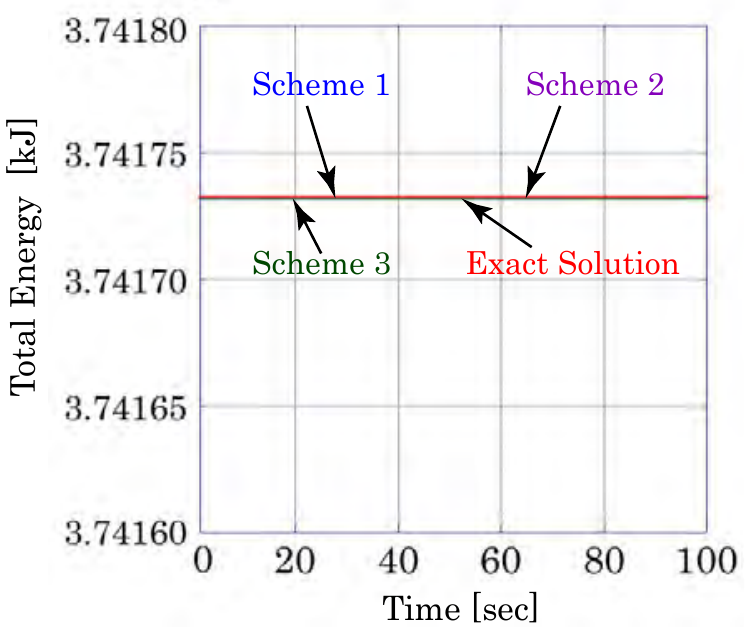}
        \end{center}
      \end{minipage}
      \qquad
      \begin{minipage}{0.4\hsize}
        \begin{center}
          \includegraphics[clip, width=5.0cm]{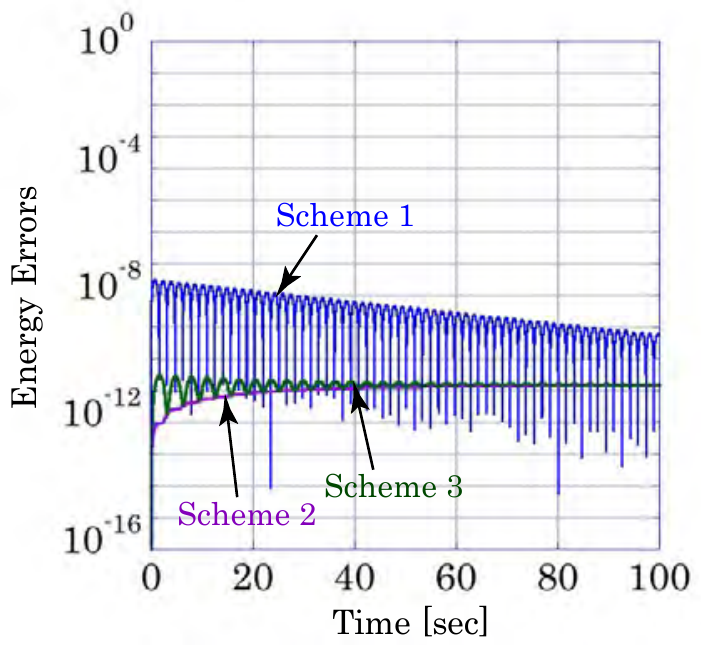}
        \end{center}
      \end{minipage}
    \end{tabular}
   \caption{Total energy and relative energy error (Case 1: $\lambda=0.2$)}
    \label{graph_energy_errors_Case1Lamda0.2}
  \end{center}
\end{figure}

\begin{figure}[htbp]
  \begin{center}
    \begin{tabular}{c}

      \begin{minipage}{0.4\hsize}
        \begin{center}
          \includegraphics[clip, width=5.3cm]{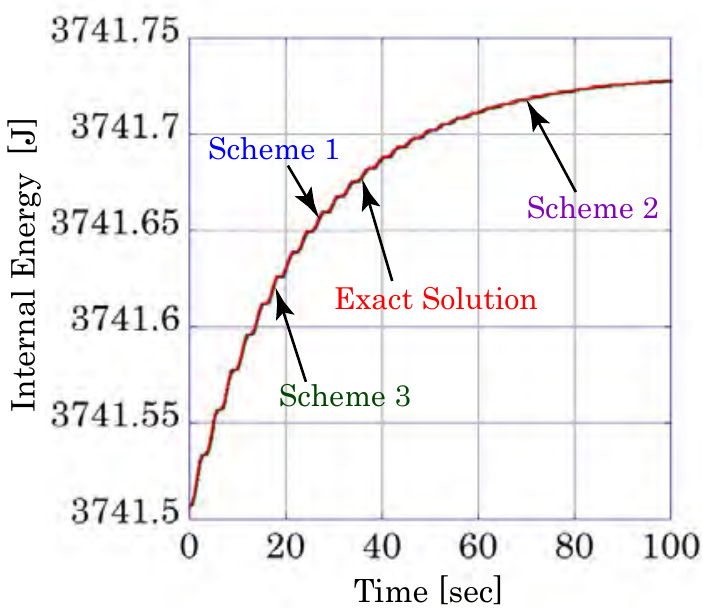}
        \end{center}
      \end{minipage}
      \qquad
      \begin{minipage}{0.4\hsize}
        \begin{center}
          \includegraphics[clip, width=5.4cm]{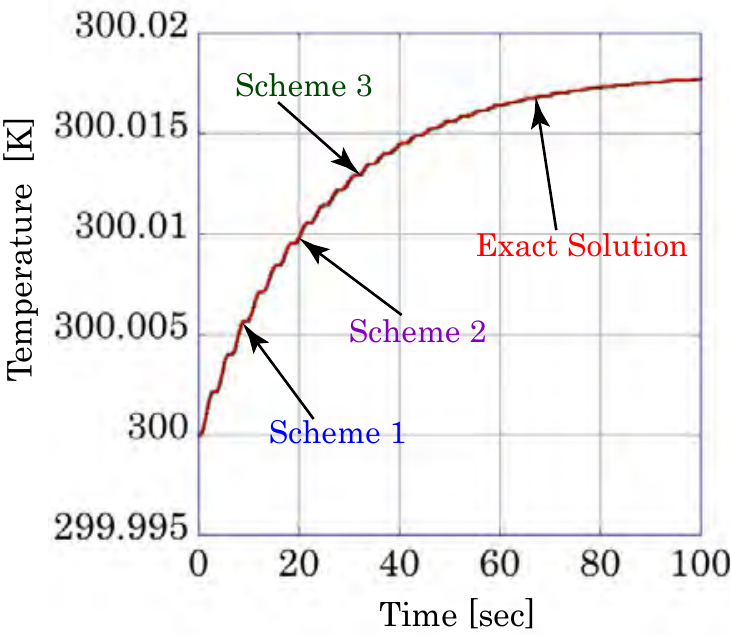}
        \end{center}
      \end{minipage}
    \end{tabular}
   \caption{Internal energy and temperature (Case 1: $\lambda=0.2$)}
    \label{graph_interenergy_temp_Case1Lamda0.2}
  \end{center}
\end{figure}

 \begin{figure}[htbp]
  \begin{center}
    \begin{tabular}{c}
      \begin{minipage}{0.4\hsize}
        \begin{center}
          \includegraphics[clip, width=5.5cm]{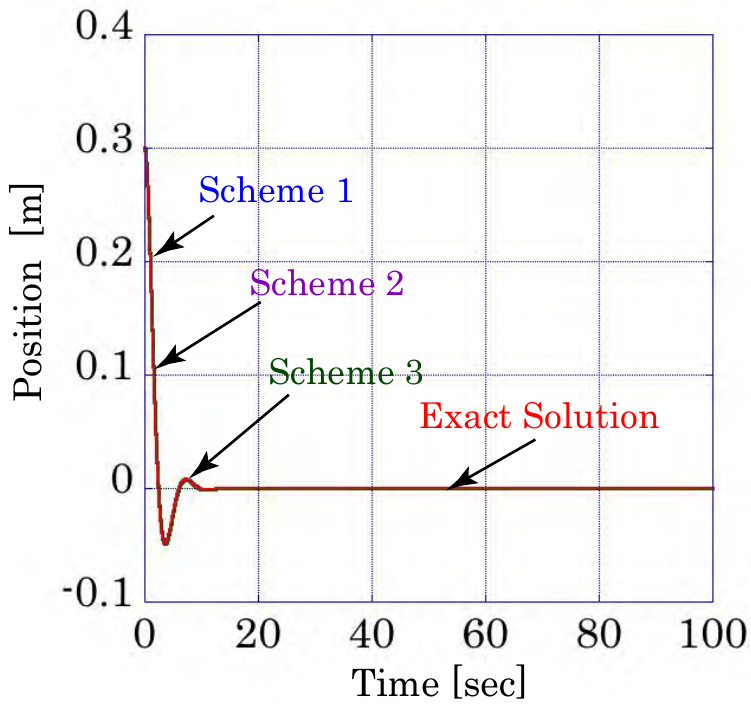}
        \end{center}
      \end{minipage}
      \qquad
      \begin{minipage}{0.4\hsize}
        \begin{center}
          \includegraphics[clip, width=6cm]{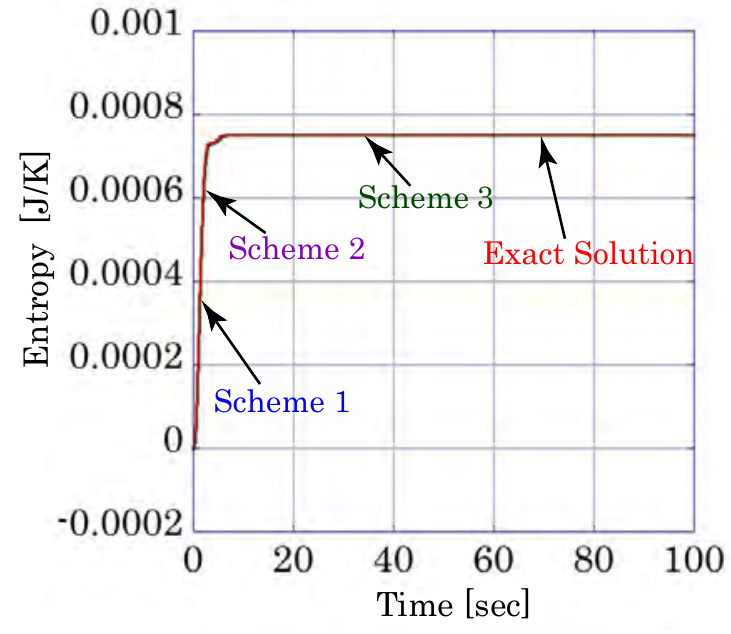}
        \end{center}
      \end{minipage}
    \end{tabular}
   \caption{Time evolutions of position and entropy (Case 1: $\lambda=5$)}
    \label{graph_position_entropy_Case1Lamda5}
  \end{center}
\end{figure}
\begin{figure}[htbp]
  \begin{center}
    \begin{tabular}{c}

      \begin{minipage}{0.4\hsize}
        \begin{center}
          \includegraphics[clip, width=6.0cm]{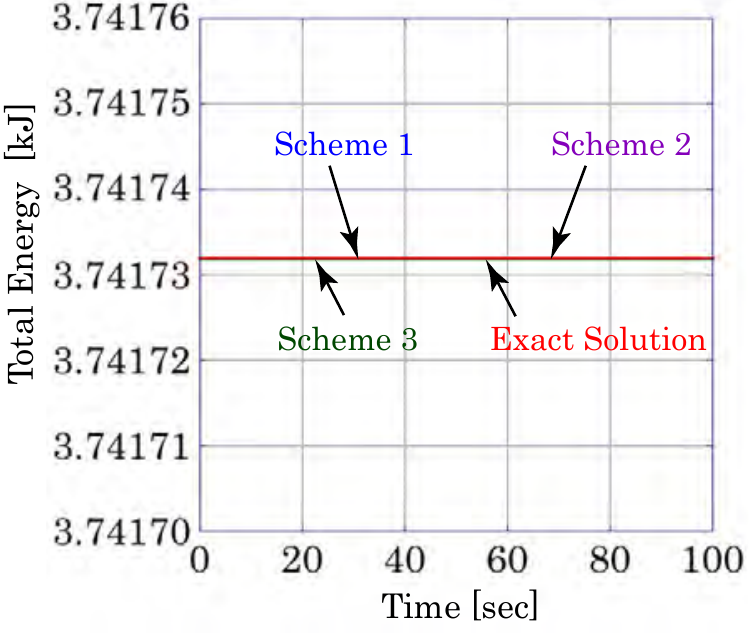}
        \end{center}
      \end{minipage}
      \qquad
      \begin{minipage}{0.4\hsize}
        \vspace{-0.1cm}\begin{center}
          \includegraphics[clip, width=5.7cm]{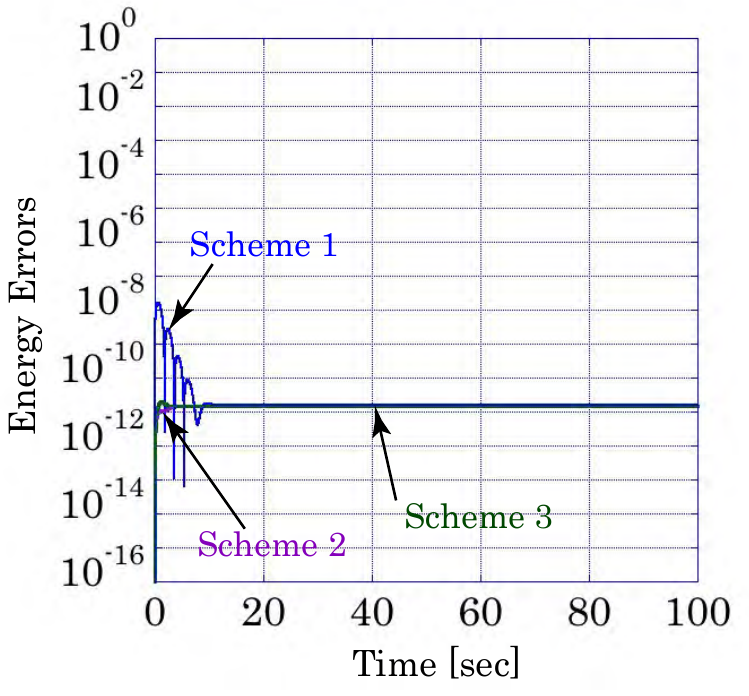}
        \end{center}
      \end{minipage}
    \end{tabular}
   \caption{Total energy and relative energy error (Case 1: $\lambda=5$)}
    \label{graph_energy_errors_Case1Lamda5}
  \end{center}
\end{figure}

\begin{figure}[htbp]
  \begin{center}
    \begin{tabular}{c}

      \begin{minipage}{0.4\hsize}
        \begin{center}
          \includegraphics[clip, width=5.6cm]{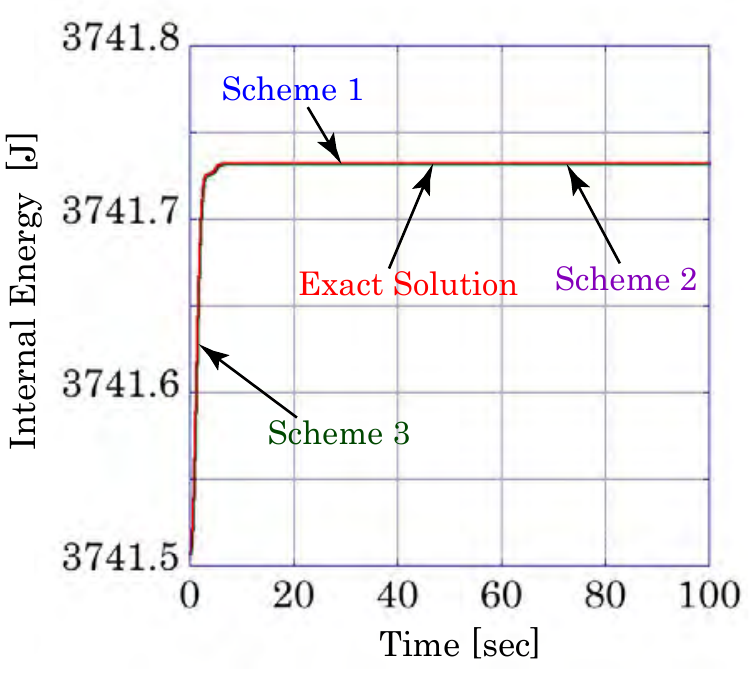}
        \end{center}
      \end{minipage}
      \qquad
      \begin{minipage}{0.4\hsize}
        \begin{center}
          \includegraphics[clip, width=5.8cm]{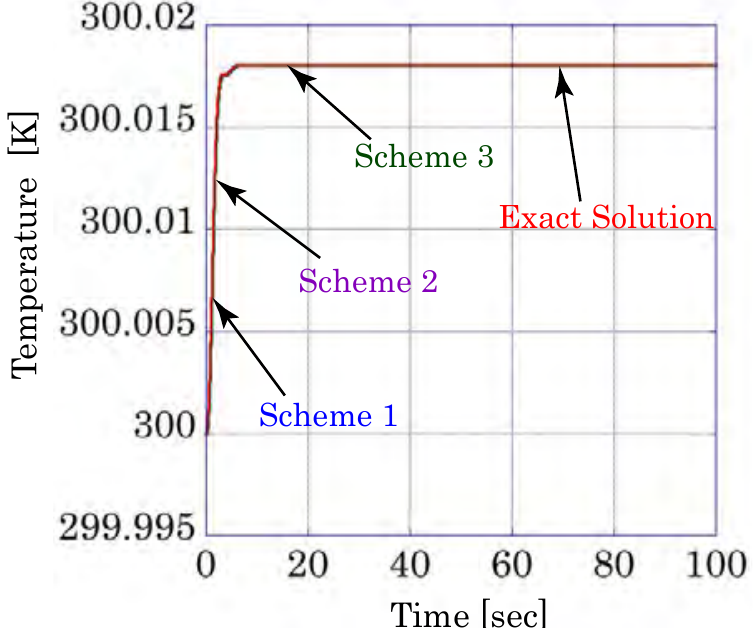}
        \end{center}
      \end{minipage}
    \end{tabular}
   \caption{Internal energy and temperature (Case 1: $\lambda=5$)}
    \label{graph_interenergy_temp_Case1Lamda5}
  \end{center}
\end{figure}


 \begin{figure}[htbp]
  \begin{center}
    \begin{tabular}{c}
      \begin{minipage}{0.4\hsize}
        \vspace{-0.1cm}\begin{center}
          \includegraphics[clip, width=5.2cm]{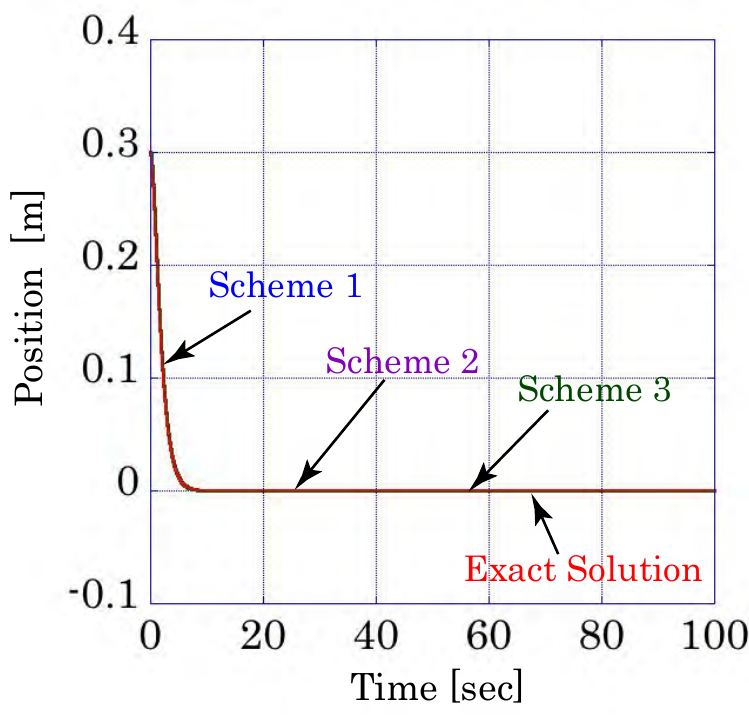}
        \end{center}
      \end{minipage}
      \qquad
      \begin{minipage}{0.4\hsize}
        \begin{center}
          \includegraphics[clip, width=5.7cm]{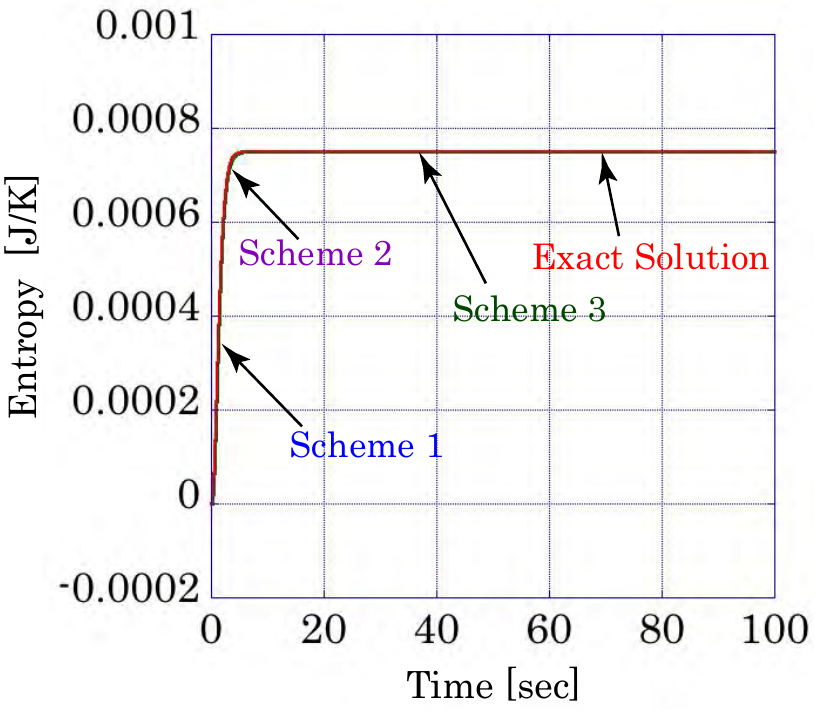}
        \end{center}
      \end{minipage}
    \end{tabular}
   \caption{Time evolutions of position and entropy (Case 1: $\lambda=10$)}
    \label{graph_position_entropy_Case1Lamda10}
  \end{center}
\end{figure}
\begin{figure}[htbp]
  \begin{center}
    \begin{tabular}{c}

      \begin{minipage}{0.4\hsize}
        \begin{center}
          \includegraphics[clip, width=6.0cm]{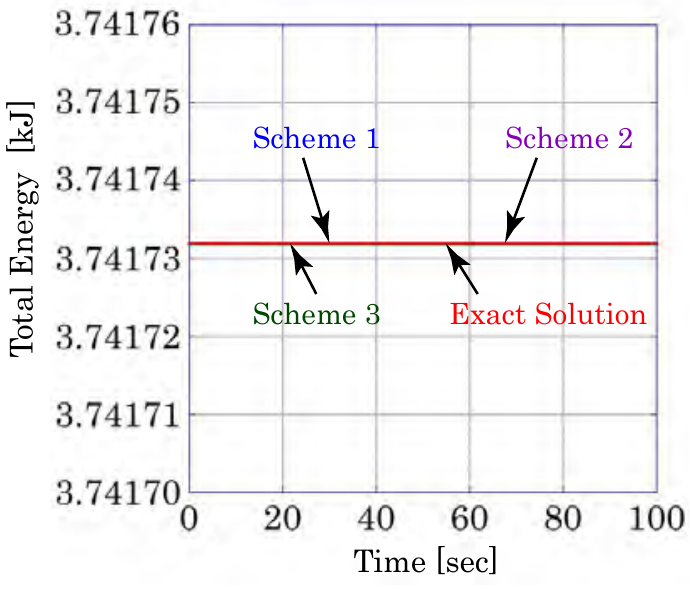}
        \end{center}
      \end{minipage}
      \qquad
      \begin{minipage}{0.4\hsize}
        \vspace{-0.1cm}\begin{center}
          \includegraphics[clip, width=5.5cm]{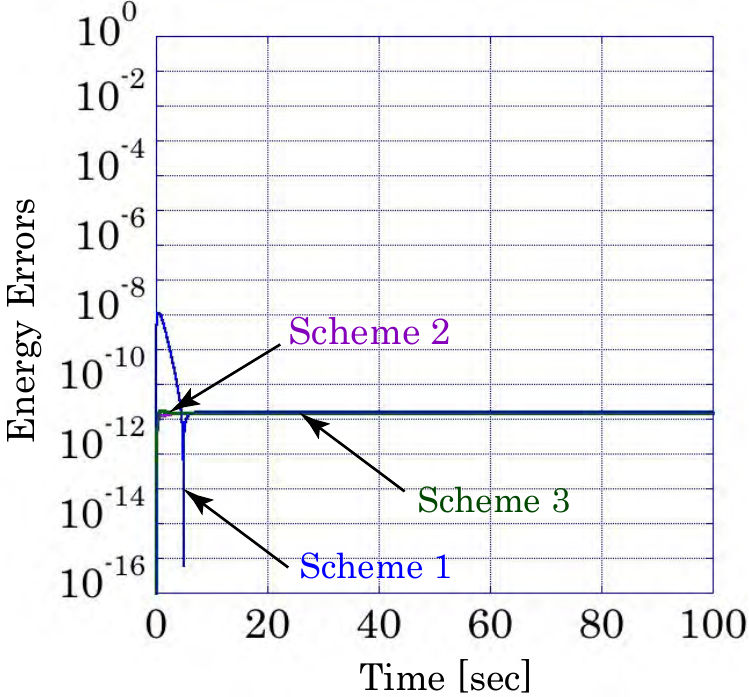}
        \end{center}
      \end{minipage}
    \end{tabular}
   \caption{Total energy and relative energy error (Case 1: $\lambda=10$)}
    \label{graph_energy_errors_Case1Lamda10}
  \end{center}
\end{figure}

\begin{figure}[htbp]
  \begin{center}
    \begin{tabular}{c}

      \begin{minipage}{0.4\hsize}
        \begin{center}
          \includegraphics[clip, width=5.6cm]{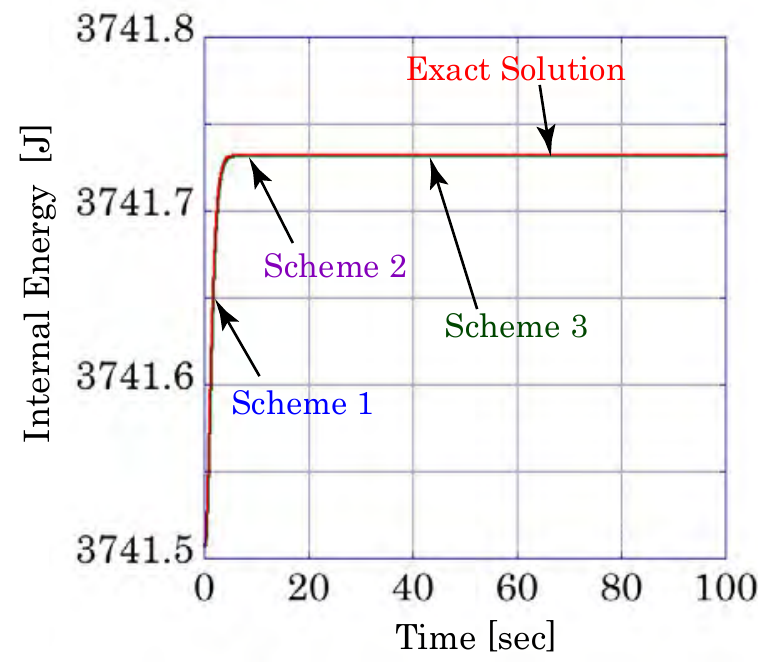}
        \end{center}
      \end{minipage}
      \qquad
      \begin{minipage}{0.4\hsize}
        \vspace{-0.8cm}\begin{center}
          \includegraphics[clip, width=6.3cm]{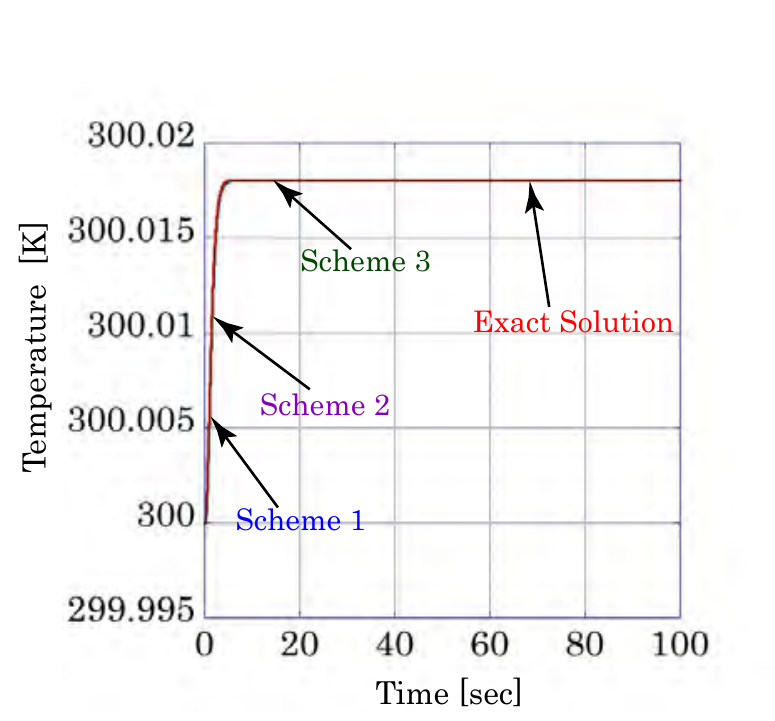}
        \end{center}
      \end{minipage}
    \end{tabular}
   \caption{Internal energy and temperature (Case 1: $\lambda=10$)}
    \label{graph_interenergy_temp_Case1Lamda10}
  \end{center}
\end{figure}


\newpage

\paragraph{Case 2.} For the second series of numerical tests, we choose the time step $h=10^{-3}{\rm [s]}$ and we set the parameters of the system $\boldsymbol{\Sigma}$ as follows: $m=10\,  {\rm[kg]}$, $N=2\,  {\rm [mol]}$,  $k= 20\, {\rm[N/m]}$, $V=9.9775 \times 10^{-2}\,  {\rm [m^{3}]}$. The initial conditions are $x_{0}=0.1\, {\rm [m]}$, $x_{1}=0.1\, {\rm [m]}$, $T_{0}=300\, {\rm [K]}$, $S_{0}=0\, {\rm [J/K]}$. \medskip

The results for $ \lambda =0$ are shown in Figures \ref{graph_position_entropy_Case2Lamda0} and \ref{graph_energy_errors_Case2Lamda0}, and, similarly to Case 1, consistently recover the behavior obtained through a usual variational discretization of the Euler-Lagrange equations. In particular, the internal energy $\mathcal{U}(S_{k})=\mathcal{U}_{0}$ is preserved and the temperature, given by $T_{k}=\frac{\partial \mathcal{U}}{\partial S}(S_{k})$,  remains a constant, see Figure \ref{graph_interenergy_temp_Case2Lamda0}.

For all the cases with friction, $ \lambda =0.2$, $5$, $10$, the numerical solutions of the position, entropy, internal energy, and temperature reproduce the correct behaviors for all the three schemes, as we see from a direct comparison with the exact solutions, see Figures \ref{graph_position_entropy_Case2Lamda0.2}, \ref{graph_interenergy_temp_Case2Lamda0.2}, \ref{graph_position_entropy_Case2Lamda5}, \ref{graph_interenergy_temp_Case2Lamda5}, \ref{graph_position_entropy_Case2Lamda10}, \ref{graph_interenergy_temp_Case2Lamda10}.

For Scheme 1, the relative energy error is bounded by $10^{-6}$ for all values of $ \lambda $, and decreases in time, whereas for Scheme 2 and 3, the relative energy error is bounded by $10^{-9}$ for all values of $ \lambda $, see Figures \ref{graph_energy_errors_Case2Lamda0.2}, \ref{graph_energy_errors_Case2Lamda5}, \ref{graph_energy_errors_Case2Lamda10}, and stays constant in time.

\begin{figure}[htbp]
  \begin{center}
    \begin{tabular}{c}
      \begin{minipage}{0.4\hsize}
        \begin{center}
          \includegraphics[clip, width=5.6cm]{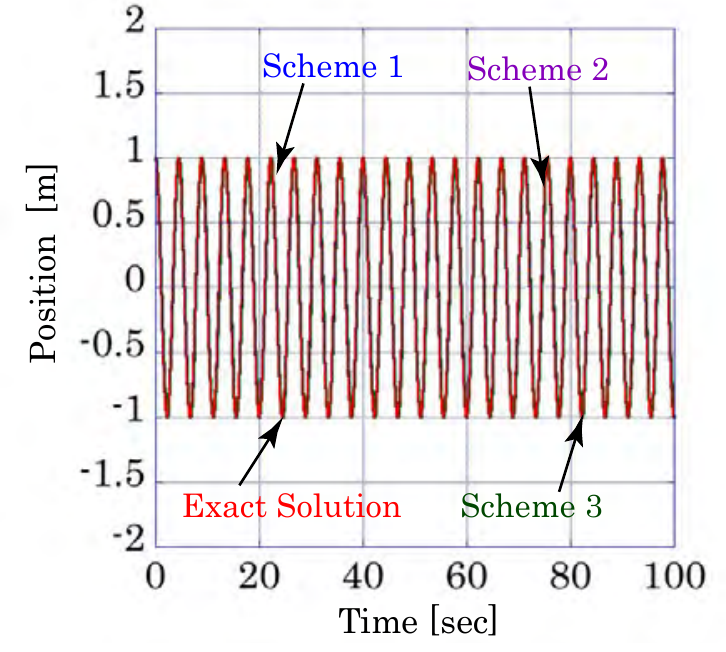}
        \end{center}
      \end{minipage}
      \qquad
      \begin{minipage}{0.4\hsize}
        \begin{center}
          \includegraphics[clip, width=5.8cm]{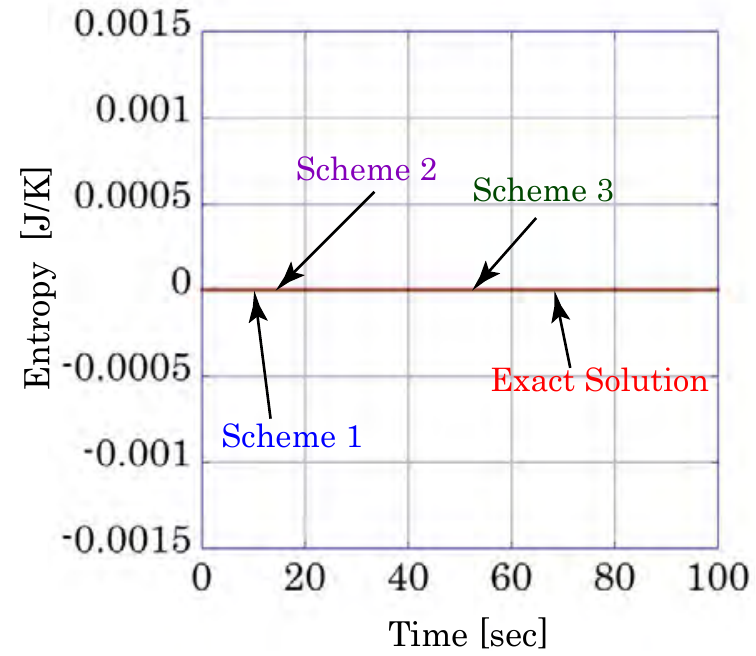}
        \end{center}
      \end{minipage}
    \end{tabular}
   \caption{Time evolutions of position and entropy (Case 2: $\lambda=0$)}
    \label{graph_position_entropy_Case2Lamda0}
  \end{center}
\end{figure}
\begin{figure}[htbp]
  \vspace{-0.5cm}\begin{center}
    \begin{tabular}{c}
      \begin{minipage}{0.4\hsize}
        \begin{center}
          \includegraphics[clip, width=6.0cm]{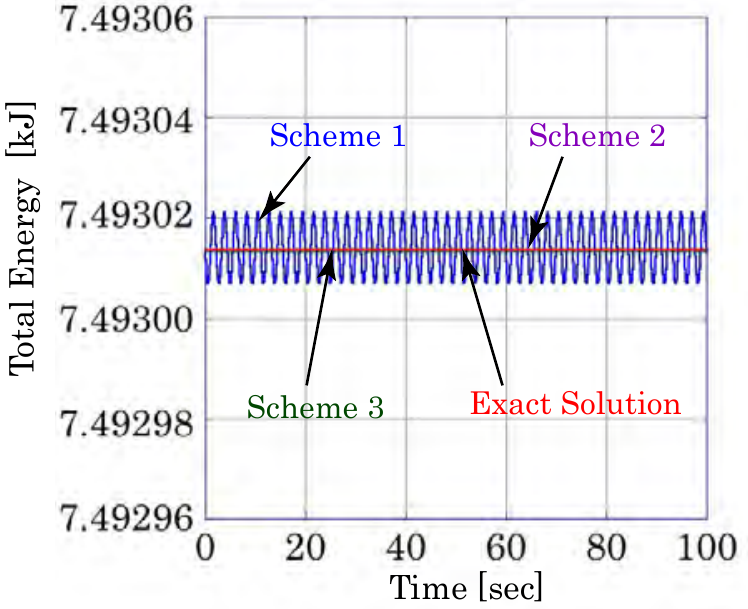}
        \end{center}
      \end{minipage}
      \qquad
      \begin{minipage}{0.4\hsize}
        \vspace{-0.1cm}\begin{center}
          \includegraphics[clip, width=5.6cm]{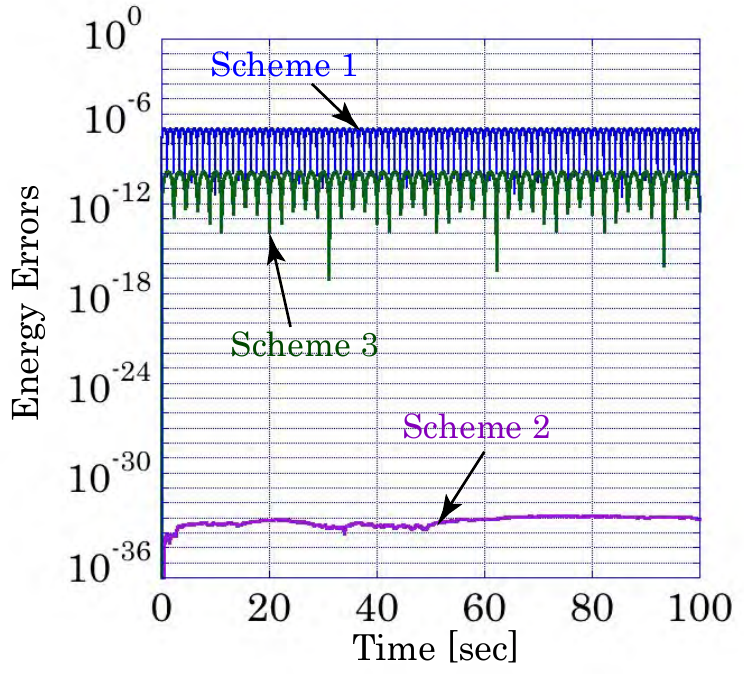}
        \end{center}
      \end{minipage}
    \end{tabular}
   \caption{Total energy and relative energy error (Case 2: $\lambda=0$)}
    \label{graph_energy_errors_Case2Lamda0}
  \end{center}
\end{figure}
\begin{figure}[htbp]
  \begin{center}
    \begin{tabular}{c}
      \begin{minipage}{0.4\hsize}
        \begin{center}
          \includegraphics[clip, width=5.4cm]{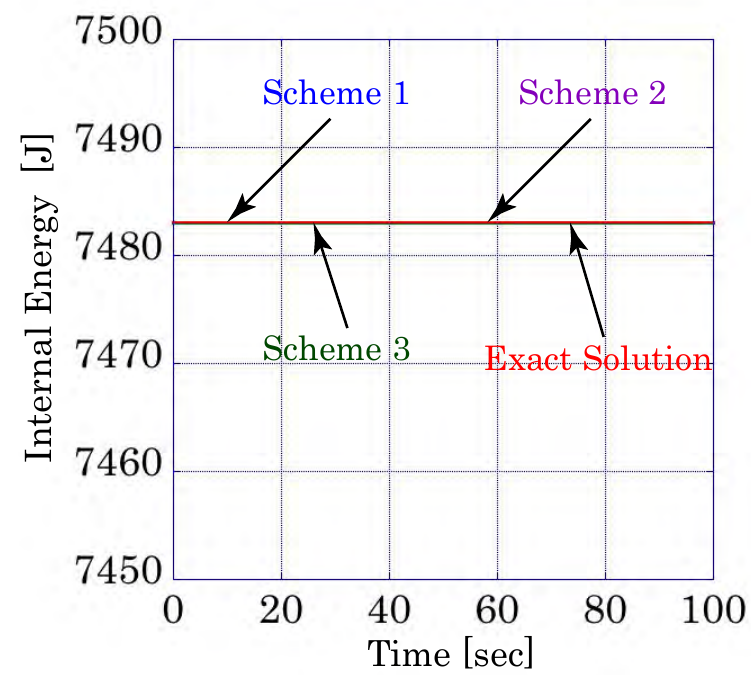}
        \end{center}
      \end{minipage}
      \qquad
      \begin{minipage}{0.4\hsize}
        \vspace{0.1cm}\begin{center}
          \includegraphics[clip, width=5.3cm]{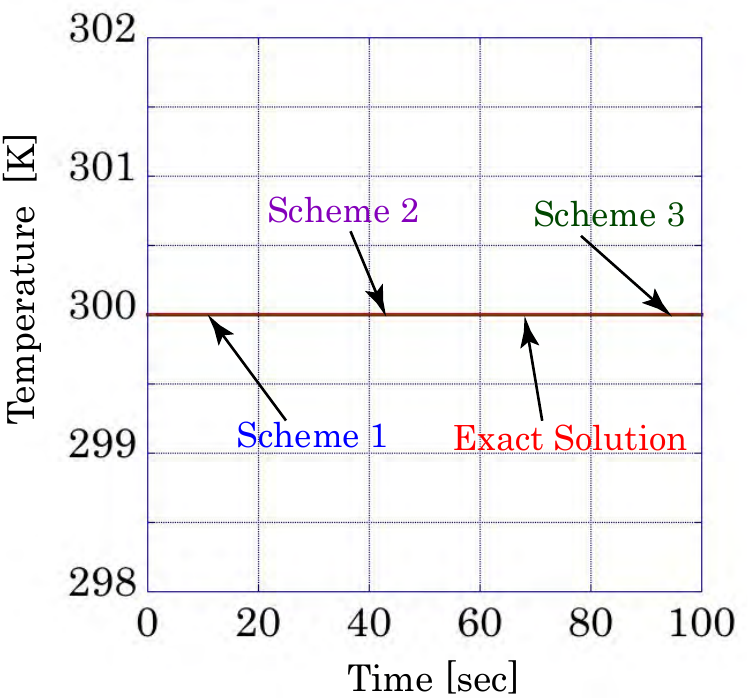}
        \end{center}
      \end{minipage}
    \end{tabular}
   \caption{Internal energy and temperature (Case 2: $\lambda=0$)}
    \label{graph_interenergy_temp_Case2Lamda0}
  \end{center}
\end{figure}


 \begin{figure}[htbp]
  \begin{center}
    \begin{tabular}{c}
      \begin{minipage}{0.4\hsize}
        \begin{center}
          \includegraphics[clip, width=5.6cm]{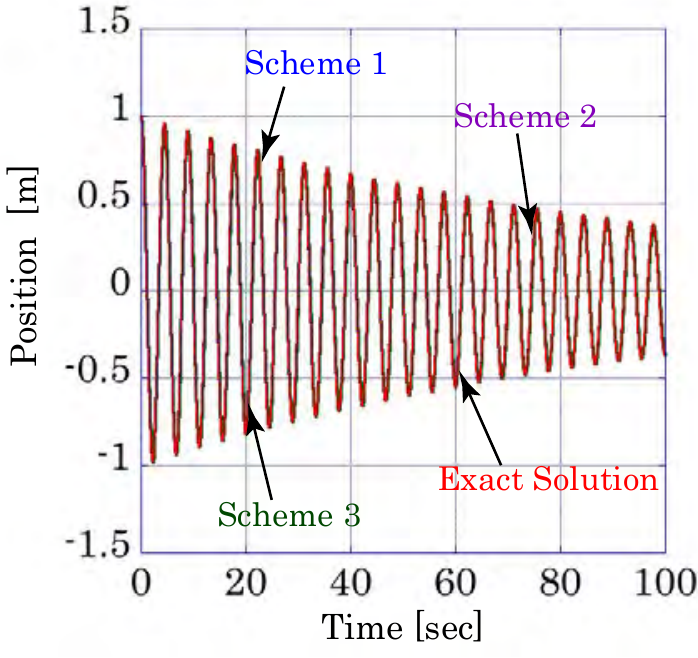}
        \end{center}
      \end{minipage}
      \qquad
      \begin{minipage}{0.4\hsize}
        \begin{center}
          \includegraphics[clip, width=5.9cm]{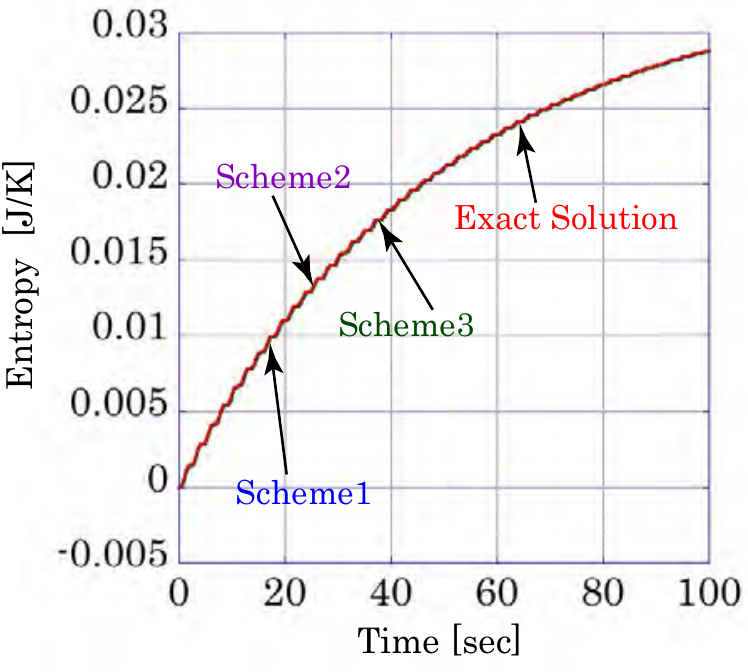}
        \end{center}
      \end{minipage}
    \end{tabular}
   \caption{Time evolutions of position and entropy (Case 2: $\lambda=0.2$)}
    \label{graph_position_entropy_Case2Lamda0.2}
  \end{center}
\end{figure}
\begin{figure}[htbp]
  \begin{center}
    \begin{tabular}{c}
      \begin{minipage}{0.4\hsize}
        \begin{center}
          \includegraphics[clip, width=6.5cm]{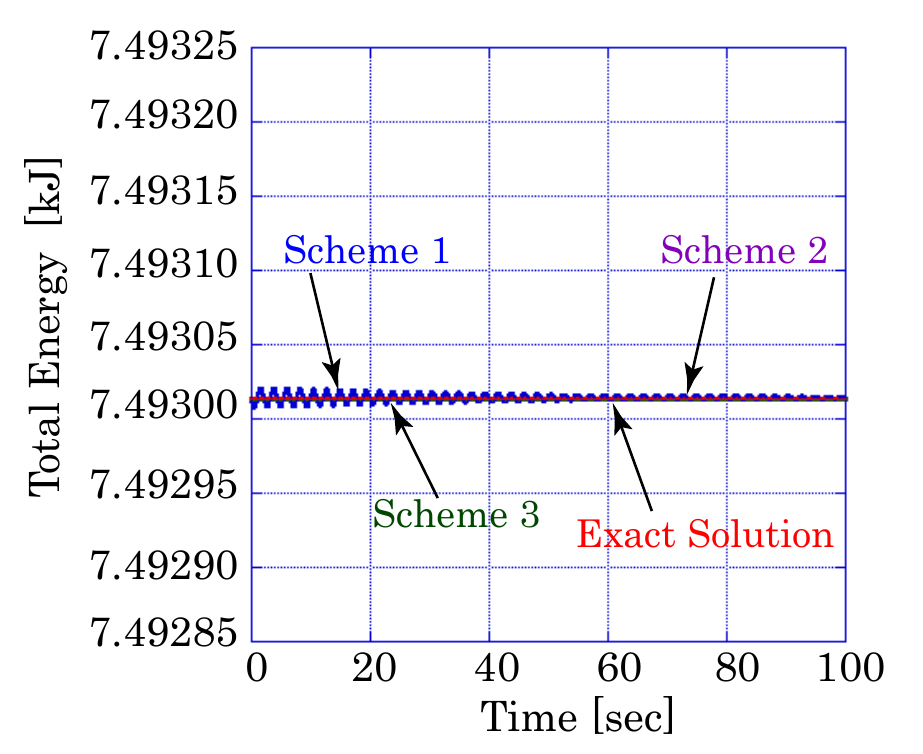}
        \end{center}
      \end{minipage}
      \qquad \qquad 
      \begin{minipage}{0.4\hsize}
        \begin{center}
          \includegraphics[clip, width=5.8cm]{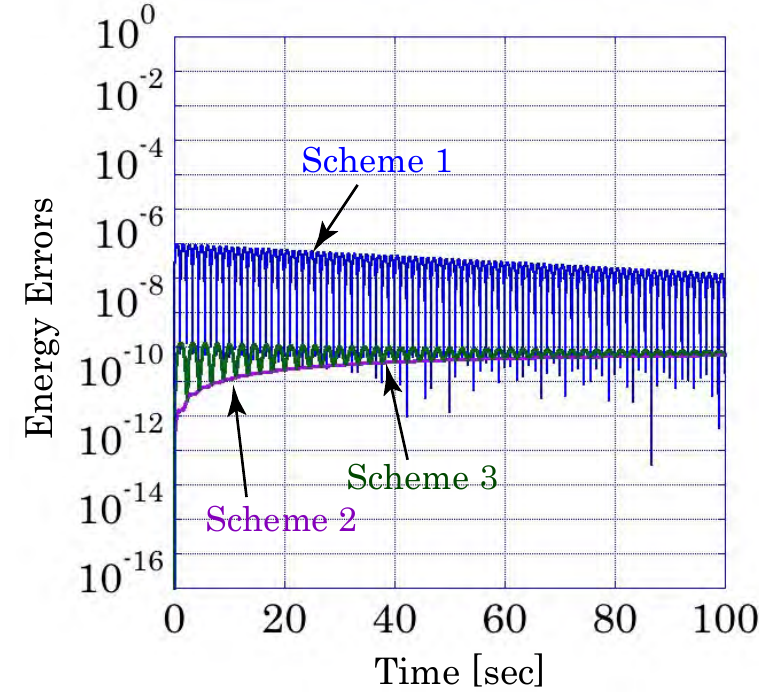}
        \end{center}
      \end{minipage}
    \end{tabular}
   \caption{Total energy and relative energy error (Case 2: $\lambda=0.2$)}
    \label{graph_energy_errors_Case2Lamda0.2}
  \end{center}
\end{figure}

\begin{figure}[htbp]
  \vspace{-0.5cm}\begin{center}
    \begin{tabular}{c}
      \begin{minipage}{0.4\hsize}
        \begin{center}
          \includegraphics[clip, width=5.2cm]{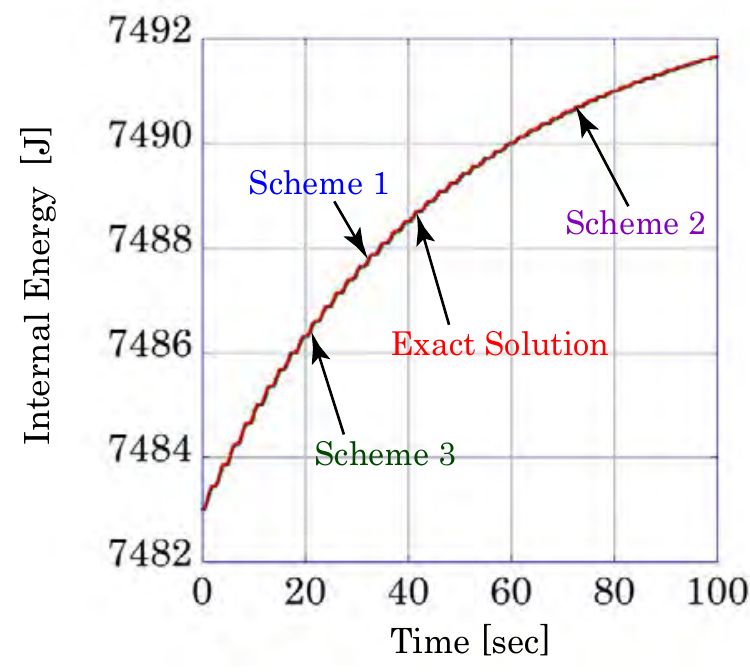}
        \end{center}
      \end{minipage}
      \qquad
      \begin{minipage}{0.4\hsize}
        \begin{center}
          \includegraphics[clip, width=5.2cm]{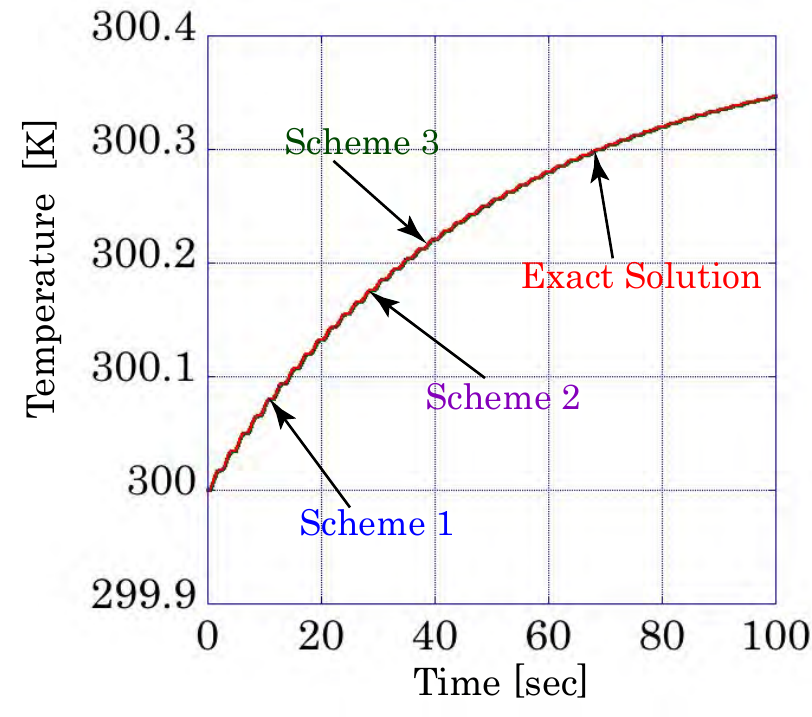}
        \end{center}
      \end{minipage}
    \end{tabular}
   \caption{Internal energy and temperature (Case 2: $\lambda=0.2$)}
    \label{graph_interenergy_temp_Case2Lamda0.2}
  \end{center}
\end{figure}
 \begin{figure}[htbp]
  \vspace{-2cm}\begin{center}
    \begin{tabular}{c}
      \begin{minipage}{0.4\hsize}
        \begin{center}
          \includegraphics[clip, width=5.7cm]{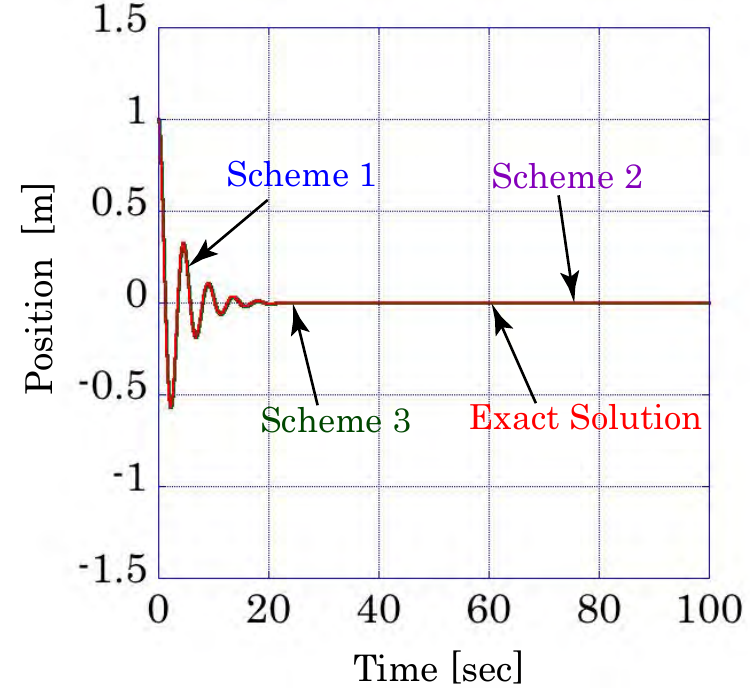}
        \end{center}
      \end{minipage}
      \qquad
      \begin{minipage}{0.4\hsize}
        \begin{center}
          \includegraphics[clip, width=5.7cm]{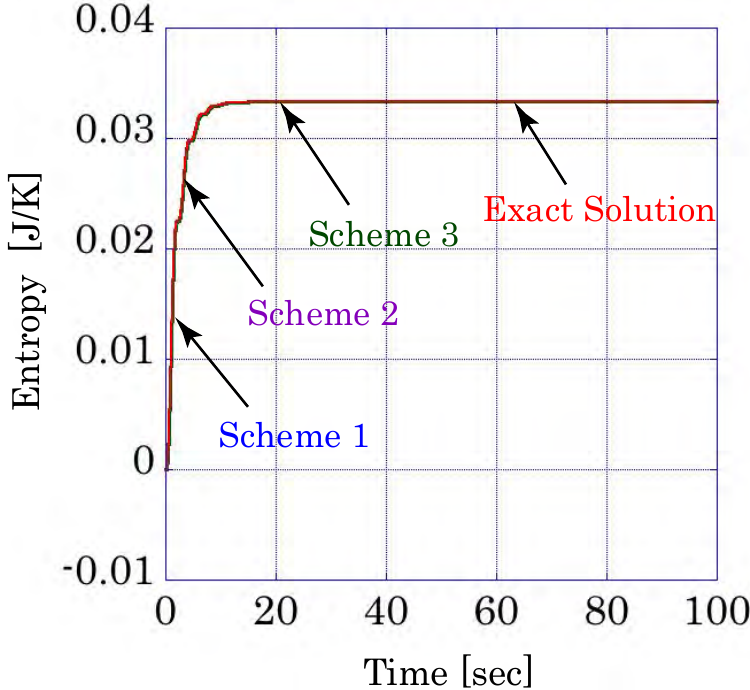}
        \end{center}
      \end{minipage}
    \end{tabular}
   \caption{Time evolutions of position and entropy (Case 2: $\lambda=5$)}
    \label{graph_position_entropy_Case2Lamda5}
  \end{center}
\end{figure}

\begin{figure}[htbp]
  \vspace{-0.5cm}\begin{center}
    \begin{tabular}{c}
      \begin{minipage}{0.4\hsize}
        \begin{center}
          \includegraphics[clip, width=6.0cm]{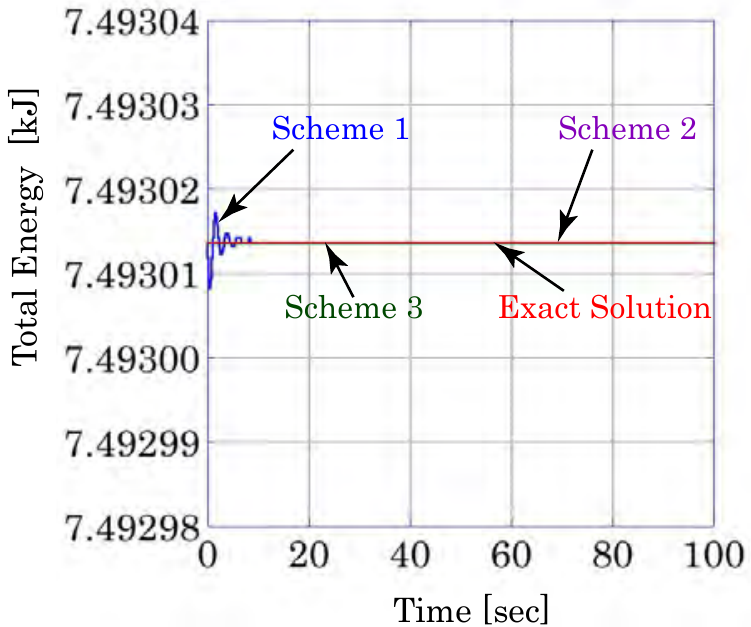}
        \end{center}
      \end{minipage}
      \qquad
      \begin{minipage}{0.4\hsize}
        \vspace{-0.15cm}\begin{center}
          \includegraphics[clip, width=5.6cm]{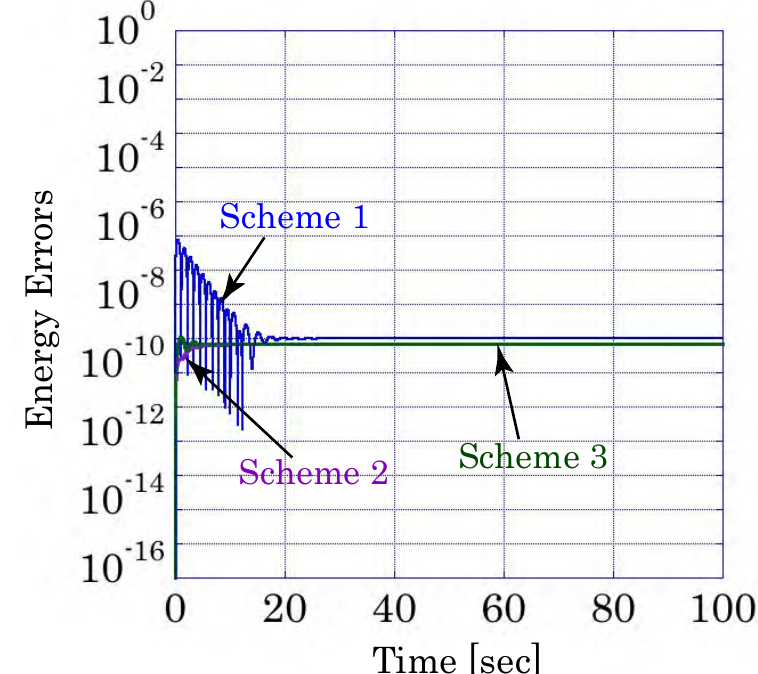}
        \end{center}
      \end{minipage}
    \end{tabular}
   \caption{Total energy and relative energy error (Case 2: $\lambda=5$)}
    \label{graph_energy_errors_Case2Lamda5}
  \end{center}
\end{figure}

\begin{figure}[htbp]
  \begin{center}
    \begin{tabular}{c}

      \begin{minipage}{0.4\hsize}
        \begin{center}
          \includegraphics[clip, width=5.6cm]{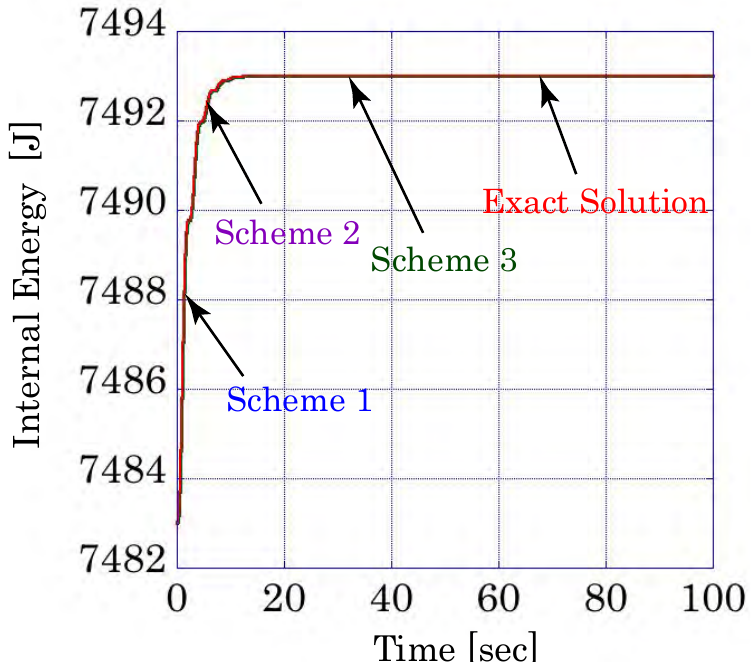}
        \end{center}
      \end{minipage}
      \qquad
      \begin{minipage}{0.4\hsize}
        \vspace{0.1cm}\begin{center}
          \includegraphics[clip, width=5.7cm]{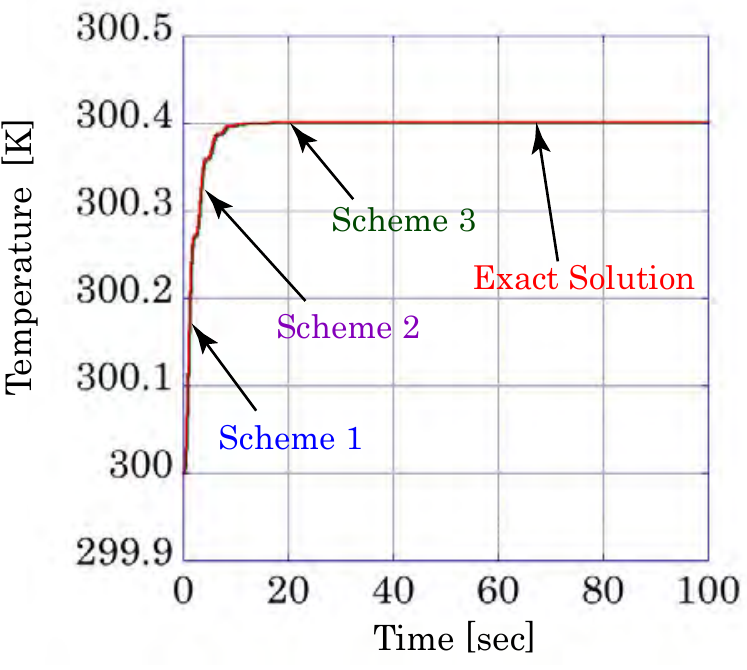}
        \end{center}
      \end{minipage}
    \end{tabular}
   \caption{Internal energy and temperature (Case 2: $\lambda=5$)}
    \label{graph_interenergy_temp_Case2Lamda5}
  \end{center}
\end{figure}


 \begin{figure}[htbp]
  \begin{center}
    \begin{tabular}{c}
      \begin{minipage}{0.4\hsize}
        \begin{center}
          \includegraphics[clip, width=5.6cm]{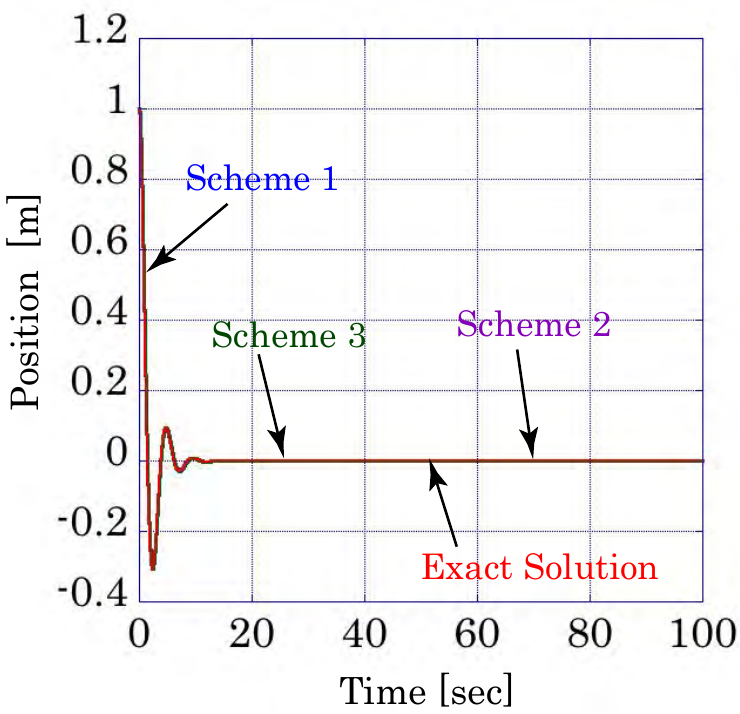}
        \end{center}
      \end{minipage}
      \qquad
      \begin{minipage}{0.4\hsize}
        \vspace{-0.15cm}\begin{center}
          \includegraphics[clip, width=5.8cm]{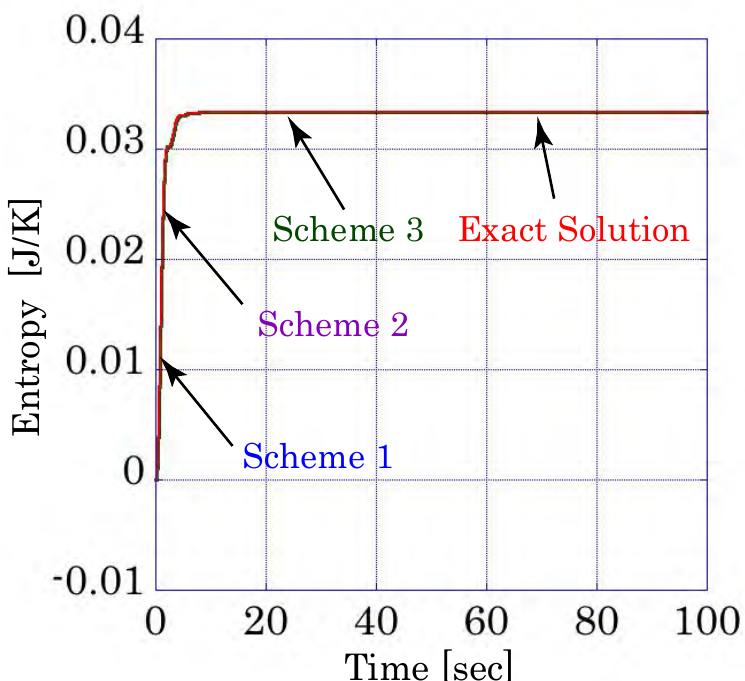}
        \end{center}
      \end{minipage}
    \end{tabular}
   \caption{Time evolutions of position and entropy (Case 2: $\lambda=10$)}
    \label{graph_position_entropy_Case2Lamda10}
  \end{center}
\end{figure}
\begin{figure}[htbp]
  \begin{center}
    \begin{tabular}{c}
      \begin{minipage}{0.4\hsize}
        \begin{center}
          \includegraphics[clip, width=6.0cm]{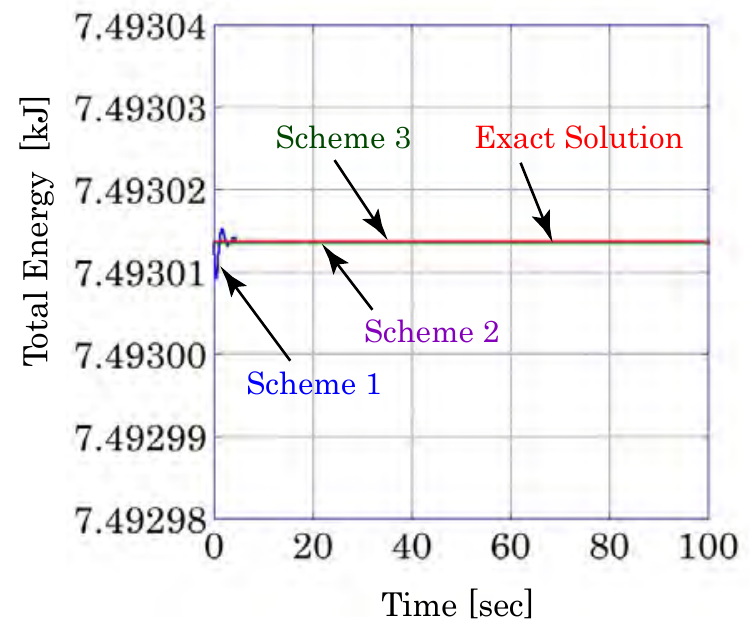}
        \end{center}
      \end{minipage}
      \qquad
      \begin{minipage}{0.4\hsize}
        \vspace{-0.25cm}\begin{center}
          \includegraphics[clip, width=5.3cm]{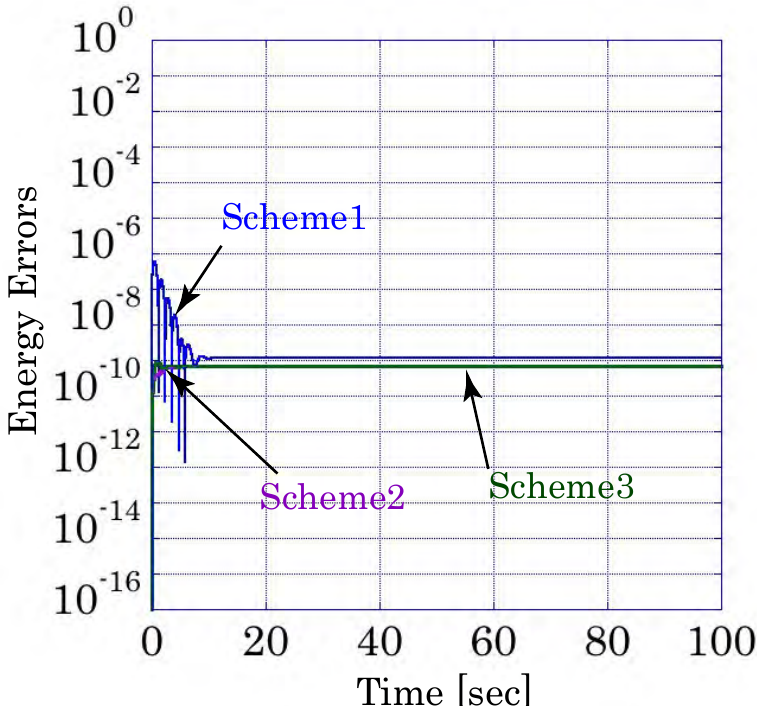}
        \end{center}
      \end{minipage}
    \end{tabular}
   \caption{Total energy and relative energy error (Case 2: $\lambda=10$)}
    \label{graph_energy_errors_Case2Lamda10}
  \end{center}
\end{figure}

\begin{figure}[htbp]
  \begin{center}
    \begin{tabular}{c}
      \begin{minipage}{0.4\hsize}
        \begin{center}
          \includegraphics[clip, width=5.6cm]{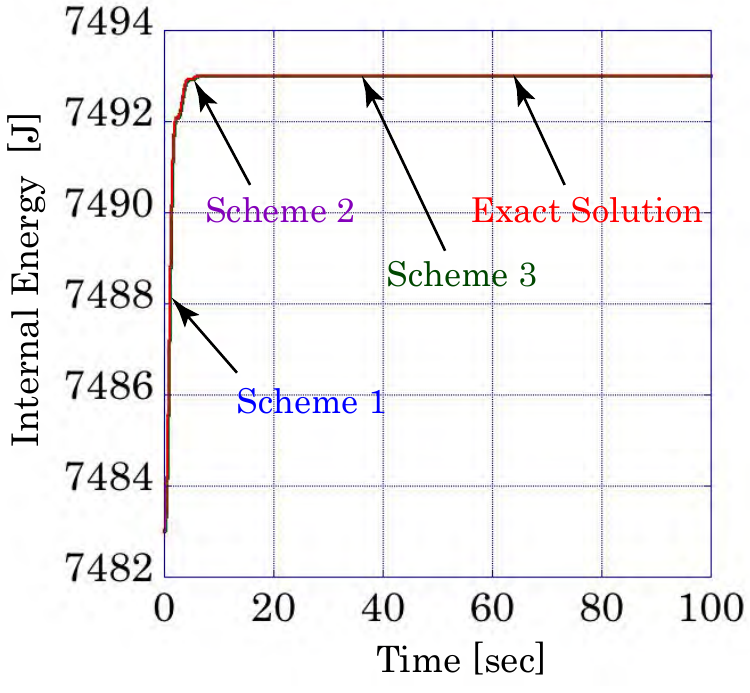}
        \end{center}
      \end{minipage}
      \qquad
      \begin{minipage}{0.4\hsize}
        \vspace{-0.1cm}\begin{center}
          \includegraphics[clip, width=5.9cm]{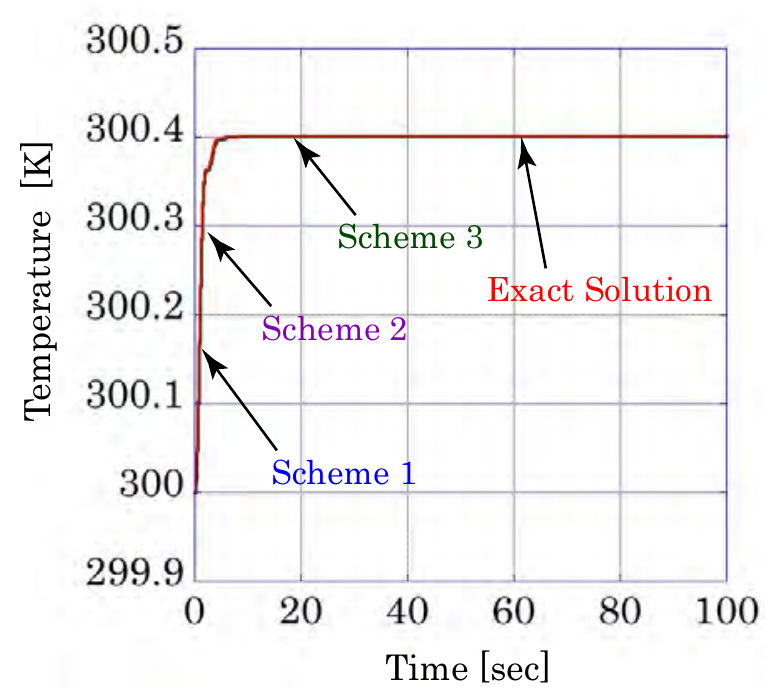}
        \end{center}
      \end{minipage}
    \end{tabular}
   \caption{Internal energy and temperature (Case 2: $\lambda=10$)}
    \label{graph_interenergy_temp_Case2Lamda10}
  \end{center}
\end{figure}


\newpage

For this particular example of the mass-spring-friction system with thermodynamics, we have observed an excellent total energy behavior for all the three schemes, i.e., changes of mechanical energy are compensated by changes of internal energy during the evolution, exactly as in the continuous case. 
A thorough study of the energy behaviors of the numerical schemes derived from our variational discretization has to be explored in order to analyze to what class of simple thermodynamical systems does this property extend.

It is important to mention that in general, a variational discretization of the Lagrange-d'Alembert type in mechanics does not necessarily produce a scheme with a well accurate energy behavior.
We refer, e.g., to \cite{MLPe2006, Celetal2016} for some examples of such schemes in nonholonomic mechanics which are derived from a discrete Lagrange-d'Alembert principle and which present an energy drift.

\color{black} 

\paragraph{Acknowledgements.} The authors thank C. Gruber for extremely helpful discussions and also graduate students, T. Nishiyama and H. Momose, for their supporting in numerical computations. F.G.B. is partially supported by the ANR project GEOMFLUID, ANR-14-CE23-0002-01; H.Y. is partially supported by JSPS Grant-in-Aid for Scientific Research (26400408, 16KT0024), Waseda University (SR 2014B-162, SR 2015B-183), and the MEXT ``Top Global University Project''.


\begin{thebibliography}{xx}



\bibitem[Bloch(2003)]{Bl2003}
Bloch, A.~M. [2003], \textit{Nonholonomic Mechanics and Control}, volume~24 of \textit{Interdisciplinary Applied Mathematics}, Springer-Verlag, New York. With the collaboration of J. Baillieul, P. Crouch and J. Marsden, and with scientific input from P. S. Krishnaprasad, R. M. Murray and D. Zenkov.

\bibitem[Celledoni, et al.(2016)]{Celetal2016}
Celledoni, E., Farr{\'e} Puiggali, M., H{\o}iseth, E.H., Martin de Diego, D.,
Energy-preserving integrators applied to nonholonomic systems,  \url{https://arxiv.org/pdf/1605.02845v1.pdf}

\bibitem[Cort\'es and Mart\'inez(2001)]{CoMa2001}
Cort\'es, J. and S. Mart\'inez [2001], Nonholonomic integrators, \textit{Nonlinearity} \textbf{14}, 1365--1392.


\bibitem[Ferrari and Gruber(2010)]{FeGr2010}
Ferrari, C. and C. Gruber [2010], Friction force: from mechanics to thermodynamics, \textit{Europ. J. Phys.} \textbf{31}(5), 1159--1175.

\bibitem[Gay-Balmaz(2017)]{GB2017}
Gay-Balmaz, F. [2017], A variational derivation of the thermodynamics of a moist atmosphere with irreversible processes, \url{https://arxiv.org/pdf/1701.03921v1.pdf}

\bibitem[Gay-Balmaz and Yoshimura(2017a)]{GBYo2016a}
Gay-Balmaz, F. and H. Yoshimura [2017a], A Lagrangian formulation for nonequilibrium thermodynamics. Part I: discrete systems, \textit{J. Geom. Phys.}, \textbf{111}, 169--193.

\bibitem[Gay-Balmaz and Yoshimura(2017b)]{GBYo2016b}
Gay-Balmaz, F. and H. Yoshimura [2017b], A Lagrangian formulation for nonequilibrium thermodynamics. Part II: continuum systems, \textit{J. Geom. Phys.}, \textbf{111}, 194--212.

\bibitem[Green and Naghdi(1991)]{GrNa1991}
Green, A. E. and P. M. Naghdi [1991], A re-examination of the basic postulates of thermomechanics, {\it Proc. R. Soc. London.}  Series A: {\it Mathematical, Physical and Engineering Sciences}, {\bf 432}(1885), 171--194.


\bibitem[Gruber(1999)]{Gr1999}
Gruber, C. [1999], Thermodynamics of systems with internal adiabatic constraints: time evolution of the adiabatic piston, \textit{Eur. J. Phys.} \textbf{20}, 259--266. 
\bibitem[Gruber and Brechet(2011)]{GrBr2011}
Gruber, C. and S.~D. Brechet [2011], Lagrange equation coupled to a thermal equation: mechanics as a consequence of thermodynamics, \textit{Entropy} \textbf{13}, 367--378. 


\bibitem[Hairer, Lubich, and Wanner(2006)]{HaLuWa2006}
Hairer, E., C. Lubich, and G. Wanner [2006],
\textit{Geometric Numerical Integration, Structure-Preserving 
Algorithms for Ordinary Differential Equations},
Springer Series in Computational Mathematics, \textbf{31}, 
Springer, Heidelberg, 2010.




\bibitem[Kane, Marsden, Ortiz, and West(2000)]{KaMaOrWe2000}
Kane, C., J.~E. Marsden, M. Ortiz, and M. West [2000], Variational integrators and the Newmark algorithm for conservative and dissipative mechanical systems, \textit{International Journal for Numerical Methods in Engineering}, \textbf{49}(10), 1295--1325.

\bibitem[Lew, Marsden, Ortiz and West(2004)]{LeMaOrWe2004}
Lew, A., J.~E. Marsden, M. Ortiz, and M. West [2004a], Variational time integrators, \textit{Internat. J. Numer. Methods Eng.}, \textbf{60}~(1), 153--212.


\bibitem[McLachlan and Perlmutter(2006)]{MLPe2006}
McLachlan, R. and M. Perlmutter [2006], Integrators for nonholonomic mechanical systems, \textit{J. Nonlin. Sci.}, \textbf{16}(4), 283--328.

\bibitem[Moser and Veselov(1991)]{MoVe1991}
Moser, J. and A.~P. Veselov [1991], Discrete versions of some 
classical integrable systems and factorization of matrix 
polynomials, \textit{Comm. Math. Phys.} \textbf{139}, 217--243.

\bibitem[Marsden and West(2001)]{MaWe2001}
Marsden, J.~E. and M. West [2001], Discrete mechanics and variational integrators, \textit{Acta Numer.}, \textbf{10}, 357--514.
   
\bibitem[Stueckelberg and Scheurer(1974)]{StSc1974}
Stueckelberg, E.~C.~G. and P.~B. Scheurer [1974], \textit{Thermocin\'etique ph\'enom\'enologique galil\'eenne}, Birkh\"auser, 1974.

\bibitem[Veselov(1988)]{Ve1988}
Veselov, A.~P. [1988], Integrable discrete-time systems and 
difference operators (Russian), \textit{Funktsional. Anal. i 
Prilozhen.}, \textbf{22}(2), 1--13, 96; English translation in 
\textit{Funct. Anal. Appl.}, \textbf{22}(2), 83--93. 


\bibitem[Veselov(1991)]{Ve1991}
Veselov, A.~P. [1991], Integrable Lagrangian correspondences and 
the factorization of matrix polynomials (Russian), 
\textit{Funkts. Anal. Prilozhen} \textbf{25}(2), 38--49; English
translation in \textit{Funct. Anal. Appl.}, \textbf{25}(2), 
112--122. 


\bibitem[von Helmholtz(1884)]{He1884}
von Helmholtz, H. [1884], Studien zur Statik monocyklischer Systeme.
\textit{Sitzungsberichte der K\"{o}niglich Preussischen Akademie der Wissenschaften
zu Berlin}, 159--177.


\bibitem[Wendlandt and Marsden(1997)]{WeMa1997}
Wendlandt, J.~M. and J.~E. Marsden [1997], \textit{Mechanical integrators derived from
a discrete variational principle}, \textit{Physica D}, \textbf{106}, 223--246.

\end{thebibliography}
\end{document}